\documentclass[12pt,reqno]{amsart}

\usepackage{amsfonts,amssymb,amsmath,amsopn,amsthm}
\usepackage{amsbsy,dsfont}
\usepackage{amsxtra, mathrsfs}

\usepackage{graphicx}
\graphicspath{ {./images/} }
\usepackage{float}
\usepackage{colordvi}
\usepackage[usenames,dvipsnames]{color}

\usepackage{xspace}
\usepackage{lipsum}
\usepackage[many]{tcolorbox}

\usepackage{a4wide,wrapfig,caption,subcaption,epstopdf}

 \numberwithin{equation}{section}

\newtcolorbox{leftbrace}{%
	enhanced jigsaw,
	breakable, 
	frame hidden, 
	overlay={%
		\draw [
		decoration={brace,amplitude=0.5em},
		decorate,
		ultra thick,
		]
		(frame.south west)--(frame.north west);
	},
	parbox=false,
}

\usepackage[utf8]{inputenc}
\usepackage{enumitem}

\usepackage{verbatim}

\usepackage{lineno}
\usepackage{textcomp}
\usepackage{mathtools,hyperref}
\usepackage{cleveref}

\usepackage{setspace,esint}
\usepackage{enumitem}

\newcommand{\Hmm}[1]{\leavevmode{\marginpar{\tiny%
			$\hbox to 0mm{\hspace*{-0.5mm}$\leftarrow$\hss}%
			\vcenter{\vrule depth 0.1mm height 0.1mm width \the\marginparwidth}%
			\hbox to
			0mm{\hss$\rightarrow$\hspace*{-0.5mm}}$\\\relax\raggedright #1}}}

\newcommand{\hf}{\frac{1}{2}}

\newcommand{\N}{\mathbb{N}}

\newcommand{\R}{\mathbb{R}}

\newcommand{\curl}{{\bf curl}}

\newcommand{\e}{{\bf E}}
\newcommand{\h}{{\bf H}}
\newcommand{\J}{{\bf J}}

\newtheorem{theorem}{Theorem}[section]

\newtheorem{cor}[theorem]{Corollary}

\newtheorem{lemma}[theorem]{Lemma}

\newtheorem{defi}[theorem]{Definition}
\newtheorem{remark}[theorem]{Remark}
\newtheorem{rem}[theorem]{Remark}

\theoremstyle{definition}

\numberwithin{equation}{section}

\newcommand{\RN}[1]{%
	\textup{\uppercase\expandafter{\romannumeral#1}}%
}

\newcommand{\intr}[1]{\mathrm{int}(#1)}

\newcommand{\be}{\begin{equation}}
	\newcommand{\ee}{\end{equation}}
\newcommand{\bea}{\begin{eqnarray}}
	\newcommand{\eea}{\end{eqnarray}}
\newcommand{\bean}{\begin{eqnarray*}}
	\newcommand{\eean}{\end{eqnarray*}}

\newlength{\wex}  \newlength{\hex}
\newcommand{\ass}[1]{Let Assumptions~\ref{assump1} hold  in a bounded Lipschitz domain $\Gw$}

\def\Gw{\Omega}              
\def\Gwh{\Omega_h}

\begin{document}

	\title[4th Order Compact - Maxwell's]{Fourth Order Accurate Compact Scheme for First-Order Maxwell's Equations}
	\author {I. Versano}
	
	\address {School of Mathematical Sciences, Tel-Aviv University, Tel-Aviv 6997801, Israel}
	
	\email {idanversano@tauex.tau.ac.il}
	
	\author {E. Turkel}
	
	\address {School of Mathematical Sciences, Tel-Aviv University, Tel-Aviv 6997801, Israel}
	
	\email {turkel@tauex.tau.ac.il}
	
	\author {S. Tsynkov }
	
	\address {North Carolina State University, Box 8205, Raleigh, NC 27695, USA.}
	
	\email {tsynkov@math.ncsu.edu}


	\begin{abstract}

		We construct a compact fourth-order scheme, in space and time,  for the time-dependent Maxwell's equations given as
		a first-order system on a staggered (Yee) grid.
		At each time step, we update the fields by solving positive definite second-order  elliptic equations.
		We  face the challenge of  finding compatible boundary conditions for these elliptic equations while maintaining a compact stencil.
		The proposed scheme is compared   computationally  with a non-compact scheme and  data-driven approach.
		
		\medskip
		

	\end{abstract}
	\maketitle
	\begin{center}
		{\bf Acknowledgments}
	\end{center}
	
	The work supported by  the US--Israel Bi-national Science Foundation (BSF) under grant \#~2020128.

	\section{introduction}
	Time-dependent Maxwell's equations have been established for over 125 years as the
	theoretical representation of electromagnetic phenomena \cite{land2,land8}. However, the accuracy requirement of high-frequency simulations remains a challenge due to a pollution effect which has been discovered firstly for the convection equation \cite{kreiss} and later on for the Helmholtz equation \cite{bayliss, dera}.
	The pollution effect states that, for a $p$ order time accurate scheme, the quantity $\omega^{p+1}h_{\tau}^p$ should remain constant in order for the error to remain constant as the wave-number $\omega$ varies.
	A similar phenomenon applies also to spatial errors \cite{bayliss,dera}.
	This effect motivates  the need for higher order schemes, i.e., larger $p$.
	Using straightforward central differences, higher-order accuracy requires a larger stencil. This has two disadvantages. Firstly, the larger the stencil, the more work may be needed to invert a matrix with  a larger bandwidth and more non-zero entries. Even more serious are the difficulties near boundaries. A large stencil requires some modification near the boundaries where all the points needed in the stencil are not available. This raises questions about the efficiency and stability of such schemes.
	In addition, higher-order compact methods can achieve the same error while using fewer grid nodes, making them more efficient in terms of both storage and CPU time than non-compact or low-order methods (see e.g., \cite{britt_tsy_tur} for the wave equation, \cite{singer_turkel} for Helmholtz equation, \cite{yefet_turkel, yefet2} for Maxwell's equations, and references therein).
	The Yee scheme (also known as finite-difference time-domain, or FDTD) was introduced by Yee \cite{yee} and remains the common numerical method used for electromagnetic simulations \cite{taflove}.
	Although it is only a  second-order method, it is still preferred for many applications because it preserves important structural features of Maxwell's equations that other methods fail to capture.
	One of the main novelties presented by Yee was the staggered-grid approach, where	the electric field $\e$ and magnetic field $\h$ do not live at the same discrete space or time locations but at separate nodes on a staggered lattice.
	From the perspective of differential forms in space-time, it becomes clear that the staggered-grid approach is more faithful to the structure of Maxwell's equations
	that are dual to one another and hence $\e, \h$ naturally live on two staggered, dual meshes
	\cite{stern}.
	The question of staggered versus non-staggered high-order schemes has been studied in \cite{gottlieb} (see also \cite[pages 63-109]{taflove}), where it has been shown that, for a given order of accuracy, a staggered scheme is more accurate and efficient than a non-staggered scheme.

	In this paper, we present a new compact fourth-order accurate scheme, in both space and time, for Maxwell's equations in three dimensions using an equation-based method on a staggered grid.
	At each time step, we update the solution by solving uncoupled (positive definite) second-order elliptic equations using the conjugate gradient method. This procedure involves a non-trivial treatment at the boundaries on which boundary conditions are not given explicitly and have to be deduced by the equation itself while maintaining the compact stencil.
	While the development of the scheme is done in 3D, the simulations are only two-dimensional for the sake of simplicity.
	\subsection{Maxwell's equations}
	
	Let $\e,\h,\bf{D},\bf{B},\bf{J}$ be (real-valued) vector fields, defined for   $(t,{\bf x})\in [0,\infty)\times\Omega$, which take values in $\R^3$.
	The Maxwell equations, in first-order differential vector form, are given by
	\begin{align}
		\label{eq"maxwell}
		\frac{\partial \bf{D}}{\partial t} &= \nabla \!\times\! \bf{H} - \bf{J}, &
		\frac{\partial \bf{B}}{\partial t}&=-\nabla \times \e  \notag \\
		\nabla \cdot \bf{H} &= 0  & \nabla \cdot \e &=\frac{ \rho}{\varepsilon} \\
		\bf{B} &= \mu \bf{H} &  \bf{D} &= \epsilon \e \notag
	\end{align}
	in $\Gw$, where in addition
	$$
	\qquad \frac{\partial \rho}{\partial t} + \nabla \cdot \bf{J} = 0.
	$$
	Here $\nabla \times $ denotes the curl operator with respect to ${\bf x}$.
	We assume that smooth initial conditions are given for $\bf{D}$ and $\bf{B}$ at $t=0$.
	We assume the following:
	\begin{enumerate}
		\item  $\mu,\varepsilon$ are positive constants, unless noted otherwise, satisfying $\frac{1}{\mu \varepsilon}=c^2$.
		\item  $\Omega=[0,1]^3$, $\vec{n} \times {\bf E}=0$ on $\partial \Omega$ (\cite[Section 8]{rolf}).
	\end{enumerate}
	The identity $\nabla \cdot (\nabla \times)=0$ implies that
	$\frac{\partial}{\partial t}\left (\nabla \cdot \h\right)=0$ and hence the divergence conditions are only restrictions on the initial conditions.
	By letting  $ \tau=ct$,  $Z=\sqrt{\mu/\varepsilon}$, we obtain the following re-scaled equations:
	\begin{subequations}
		\label{eq:maxwellv}
		\begin{align}
			\frac{\partial {\bf E}}{\partial \tau}=Z(\nabla \times \bf{H}-{\bf J}) \quad \mathrm{in} \quad  (0,\infty)\times \Omega & \label{maxe} \\
			\frac{\partial {\bf H}}{\partial \tau}=-\frac{1}{Z}\nabla \times \e \quad \mathrm{in} \quad  (0,\infty)\times \Omega & \label{maxh} \\
			\nabla \cdot \h(\tau,x,y,z) = 0 & \label{maxdivh} \\
			\nabla \cdot \e(\tau,x,y,z) =\frac{ \rho}{\varepsilon}  & \label{maxdive} \\
			\frac{1}{c}\frac{\partial \rho}{\partial \tau} + \nabla \cdot \bf{J} = 0. & \label{maxrho}
		\end{align}
	\end{subequations}
	
	This is supplemented by the initial conditions
	$$
	\e(0,x,y,z)=\e_0, \quad
	\h(0,x,y,z)=\h_0
	$$
	and the boundary conditions $\vec{n}\times {\bf E}=0$ on $\partial \Gw$.
	
	\begin{rem}
		\label{rem:wave_e}
		Differentiating \eqref{eq:maxwellv} with respect to $\tau$, we see  that $\e$ and $\h$ satisfy the wave equation
		\begin{align} \label{secondorder}
			\frac{\partial^2 \e}{\partial \tau^2}&=-Z\frac{\partial \J}{\partial \tau}-\nabla\times \nabla\times \e\\ \notag
			&=-Z\frac{\partial \J}{\partial \tau}-\nabla(\nabla \cdot \e)+\Delta \e\\
			&=-Z\frac{\partial \J}{\partial \tau}-\nabla(\nabla \cdot \e)+\Delta \e \notag
		\end{align}
		and
		\begin{equation}
			\label{secondorderh}
			\frac{ \partial^2 \h}{\partial \tau^2}=-\nabla \times \nabla \times\h-\nabla \times \partial_{\tau}\J=
			\Delta \h-\nabla \times \partial_{\tau}\J.
		\end{equation}
		Here $\Delta$ denotes the Laplacian with respect to ${\bf x}$.
	\end{rem}
	To determine the number of imposed boundary conditions, we need to examine the eigenvalues for the reduced one-dimensional
	equations at that boundary.  A positive eigenvalue indicates a variable that is entering from outside the domain
	and so needs to be specified. A negative eigenvalue indicates the variable is determined from the inside and so
	cannot be imposed and instead is determined by the interior equations. A zero eigenvalue is more ambiguous.
	\begin{remark}\label{rem:Neumann}
		An imposed boundary condition must contain only derivatives of a lower order than the differential equation.
		Thus, for a first-order system, all boundary conditions must be of Dirichlet type. A Neumann condition is
		\textbf{not} considered an imposed boundary condition.
	\end{remark}
	To derive the boundary  conditions for $\e$ and $\h$, we analyze the  characteristics of the homogeneous counterpart to system \eqref{eq"maxwell}, $\rho=0$ and $\J=0$, in the direction normal to the boundary (cf.\ \cite[Section 9.1]{GKO}).
	The two divergence equations  can be considered as constraints on the initial data and  do not affect the characteristics.
	Let $w$ be the vector $(E_x,E_y,E_z,H_x,H_y,H_z)^{t}$. Then, considering only the $x$ space derivatives (without loss of generality) we arrive at
	\begin{equation*}
		\frac{\partial w}{\partial t} = A w_x
	\end{equation*}
	where
	\begin{equation*}
		A =
		\begin{pmatrix}
			0             & 0              & 0 & 0 & 0                   & 0 \\
			0             & 0              & 0 & 0 & 0                   & \frac{1}{\epsilon} \\
			0             & 0              & 0 & 0 & -\frac{1}{\epsilon} & 0 \\
			0             & 0              & 0 & 0 & 0                   & 0 \\
			\frac{1}{\mu} & 0              & 0 & 0 & 0                   & 0 \\
			0             & -\frac{1}{\mu} & 0 & 0 & 0                   & 0
		\end{pmatrix}
	\end{equation*}
	The eigenvalues of $A$ are $ c \!=\pm\! \frac{1}{\epsilon \mu} $ (twice) and $0$ (twice). Thus, two boundary conditions need to be imposed,
	two are determined from the interior and two are not clear.
	We shall see that we impose ${\bf E}_{\parallel} = 0$ (two conditions) and ${\bf H}_{\perp} = 0$ (one condition).

	The structure of the paper is as follows:
	In Section \ref{sec:prelimi}  we introduce the background.
	In Section \ref{sec:the scheme}  we present our scheme, and Section \ref{sec:stability} is devoted to stability analysis. Section \ref{sec:TE} is devoted to numerical examples.

	\section{preliminaries}\label{sec:prelimi}
	We consider a uniform discretization of $\Gw=[0,1]^3$, $\Delta x \!=\! \Delta y \!=\! \Delta z \!=\! h$
	and $h_{\tau} $ is the time-step.
	We introduce the following notations:
	\begin{itemize}
		
		\item ${\bf x}=(x,y,z)\in \R^3$.\\[1mm]
		\item $r=\frac{h_{\tau}}{h}$ is the Courant-Friedrichs-Lewy (CFL) number. \\[1mm]
		\item $\e(t,{\bf x})=(E_x(t,{\bf x}),E_y(t,{\bf x}),E_z(t,{\bf x}))$, $\h(t,{\bf x})=(H_x(t,{\bf x}),H_y(t,{\bf x}),H_z(t,{\bf x}))$.\\[1mm]
		\item  $\Delta$ denotes the Laplacian with respect to ${\bf x}$. \\[1mm]
		\item $\Delta \e=(\Delta E_x, \Delta E_y, \Delta E_z)$,
		$\Delta \h=(\Delta H_x, \Delta  H_y, \Delta  H_z)$.
		\\[1mm]
		\item $\e^n$  (resp. $\h^{n+1/2}$) denotes $\e(t=nh_{\tau})$ (resp. $\h(t=(n+1/2)h_{\tau})$).\\[1mm]
		\item For $i,j,k\in \N \cup \{0\}$, $x_{i}=i h,y_{j}=j h,z_{k}=k h$. \\[1mm]
		\item $x_{i+\frac{1}{2}}=\left (i+\frac{1}{2}\right ) h, y_{j+\frac{1}{2}}=\left (j+\frac{1}{2}\right ) h, z_{k+\frac{1}{2}}=\left (k+\frac{1}{2}\right ) h$. \\[1mm]
		\item 	$[N]=\{0,1,..,N\},\quad N\in \N.$\\[1mm]
		\item 	$[k,m]=\{k,k+1,..m\},\quad k<m\in \N.$\\[1mm]
		\item {\bf 1} denotes the identity operator. \\[1mm]
		\item For any  operator  $T:\R^M\to \R^N$ , $\|T\|:=\sup\limits_{0\neq x\in \R^N}\frac{|Tx|}{|x|}$.
		\item  $\partial_{\tau}:=\frac{\partial}{\partial \tau}$ and $\partial_{\tau}^2:=\frac{\partial^2}{\partial^2 \tau}$.\\[1mm]
		\item $\delta_{\tau} U^{n+\frac{1}{2}}:=\frac{U^{n+1}-U^{n}}{h_{\tau}}$.\\[1mm]
		
	\end{itemize}
	To discretize the equations, we introduce a staggered mesh in both space and time as in the Yee scheme \cite{yee}.
	With this arrangement, all space derivatives are spread over a single mesh width, and the central time and space derivatives
	are centered at the same point, similar to that of the Yee scheme.
	$\e$ is  evaluated at time $nh_{\tau}$ while $\bf{H}$ and $\bf{J}$ are evaluated at time $(n\!+\! \frac{1}{2})h_{\tau}$.
	For the spatial discretization, we define the following meshes:
	\begin{align}\label{eq:meshes}
		&\Gwh^{E_z}:=\{(x_i,y_j,z_{k+\frac{1}{2}}),\quad (i,j,k)\in [N]^2 \times[N-1]\} \\ \nonumber
		& \Gwh^{E_y}:=\{(x_i,y_{j+\frac{1}{2}},z_{k}), \quad (i,j,k)\in [N]\times [N-1] \times[N] \}\\ \nonumber
		& \Gwh^{E_x}:=\{ (x_{i+\frac{1}{2}},y_{j},z_{k}),\quad (i,j,k) \in [N-1]\times [N]^2 \} \\ \nonumber
		&\Gwh^{H_z}:=\{ (x_{i+\frac{1}{2}},y_{j+\frac{1}{2}},z_{k}),\quad (i,j,k)\in [N-1]^2\times[N] \}\\
		&\Gwh^{H_y}:=\{ (x_{i+\frac{1}{2}},y_j,z_{k+\frac{1}{2}}),\quad (i,j,k) \in [N-1]\times [N] \times [N-1] \}\\ \nonumber
		& \Gwh^{H_x}:=\{ (x_i,y_{j+\frac{1}{2}},z_{k+\frac{1}{2}}), \quad (i,j,k) \in [N]\times [N-1]^2\}.\nonumber
	\end{align}
	For each $s=x,y,z$ we define the interior meshes:
	$$\mathrm{int}(\Gwh^{E_s}):=(0,1)^3\cap\Gwh^{E_s}, \qquad
	\mathrm{int}(\Gwh^{H_s}):=(0,1)^3\cap\Gwh^{H_s}.
	$$
	\begin{figure}[!t]
		\centering
		\includegraphics[scale=0.7]{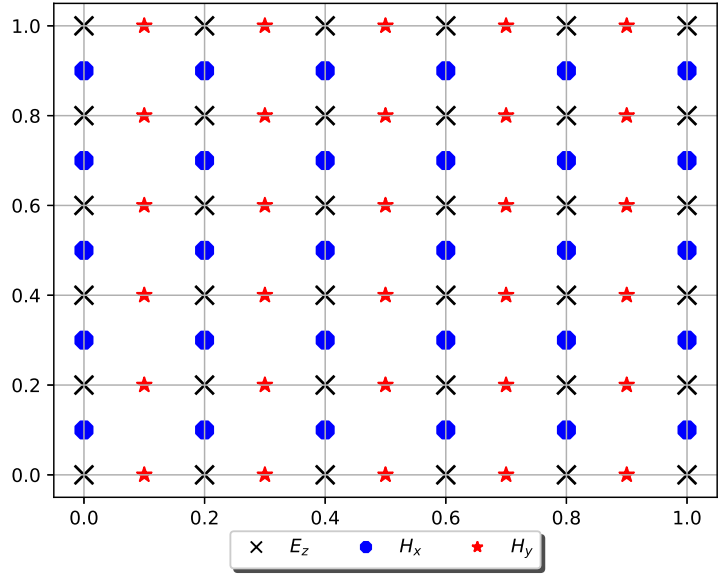}
		\caption{The grid points of $\Gwh^{E_z}, \Gwh^{H_x}, \Gwh^{H_y}$ projected on the $(x,y)$ plane in the case $N=5$.}
		\label{fig:omega}
	\end{figure}
	By convention, $x_0=y_0=z_0=0$ and $x_N=y_N=z_N=1$.
	We denote
	$$
	\Gwh^{{\bf E}}:=\Gwh^{E_x}\times \Gwh^{E_y}\times \Gwh^{E_z}, \qquad
	\Gwh^{\bf H}:=
	\Gwh^{H_x}\times \Gwh^{H_y}\times \Gwh^{H_z},
	$$
	and
	$$
	\mathrm{int}(\Gwh^{{\e}}):=\mathrm{int}(\Gwh^{E_x})\times \mathrm{int}(\Gwh^{E_y})\times \mathrm{int}(\Gwh^{E_z}),
	\quad \partial \Gwh^\e:=\Gwh^\e\setminus  \mathrm{int}(\Gwh^{\e})
	$$
	$$
	\mathrm{int}(\Gwh^{\h}):=
	\mathrm{int}(\Gwh^{H_x})\times \mathrm{int}(\Gw^{H_y})\times \mathrm{int}(\Gwh^{H_z}),
	\quad \partial \Gwh^\h:=\Gwh^\h\setminus  \mathrm{int}(\Gw^{\h})
	$$
	see Figure \ref{fig:omega} for illustration of the sets
	$\Gwh^{E_z}, \Gwh^{H_x},\Gwh^{H_y}$ projected on the $x,y$ plane.
	The discretized functions $\e_h, \h_h$ are then defined as:
	$$
	\e_h^n:=\e(nh_{\tau},\Gwh^\e),\quad  \h_h^{n+\frac{1}{2}}:=\h\left ((n+\frac{1}{2})h_{\tau},\Gwh^\h\right).
	$$
	With this arrangement, the boundary condition
	$\vec{n}\times {\bf E}=0$ on $\partial \Gwh^{{\bf E}}$ implies that
	\begin{align}\label{eq:bc_E=0}
		&\nonumber
		E_x(x_{i+\frac{1}{2}},0,z_j)=E_x(x_{i+\frac{1}{2}},N,z_j)=
		E_x(x_{i+\frac{1}{2}},y_j,0)=E_x(x_{i+\frac{1}{2}},y_j, N)=0, \\&
		E_y(0,y_{j+\frac{1}{2}},z_j)=E_y(N,y_{j+\frac{1}{2}},z_j)=
		E_y(x_i,y_{j+\frac{1}{2}},0)=E_y(x_i, y_{j+\frac{1}{2}}, N)=0, \\&
		E_z(0,y_j,z_{k+\frac{1}{2}})=E_x(N,y_j,z_{k+\frac{1}{2}})=
		E_z(x_i,0,z_{k+\frac{1}{2}})=E_z(x_i,N,z_{k+\frac{1}{2}})=0. &\nonumber
	\end{align}
	In short, $E_x\!=\!0$ when $y\!=\!0,1$ and $z\!=\!0,1$; $E_y\!=\!0$ when $x\!=\!0,1$ and $z\!=\!0,1$;$E_z\!=\!0$ when $x\!=\!0,1$ and $y\!=\!0,1$.
	
	\subsection{Boundary conditions for time derivatives}
	
	We will now derive  the boundary conditions for the time derivatives of the fields $\e$ and $\h$ rather than the fields themselves.
	We do so since time marching with our compact scheme derived in Section~\ref{sec:the scheme} involves solving elliptic equations for the time derivatives of the fields (see equation \eqref{eq:mhe} and \eqref{eq:mhh}).
	
	By \eqref{eq:bc_E=0}, we have the following Dirichlet-type boundary conditions:
	\begin{equation}
		\label{eq:BC_E}
		\begin{gathered}
			\partial_\tau E_x\!=\!0 \ \text{if} \ y\!=\!0,1 \ \text{and} \ z\!=\!0,1,\\
			\partial_\tau E_y\!=\!0 \ \text{if} \ x\!=\!0,1 \ \text{and} \ z\!=\!0,1,\\
			\partial_\tau E_z\!=\!0 \ \text{if} \ x\!=\!0,1 \ \text{and} \ y\!=\!0,1.
		\end{gathered}
	\end{equation}
	Moreover, differentiating the Gauss law of electricity $\nabla \cdot \e=\frac{\rho}{\varepsilon}$ with respect to time and using \eqref{eq:bc_E=0} we see that,
	for the remaining variables the Neumann-type conditions hold:
	\begin{gather*}
		\partial_x\partial_{\tau}E_x= \dfrac{\partial_\tau\rho}{\varepsilon} \ \text{if} \ x=0,1,\\
		\partial_y\partial_{\tau}E_y= \dfrac{\partial_\tau\rho}{\varepsilon} \ \text{if} \ y=0,1,\\
		\partial_z\partial_{\tau}E_z= \dfrac{\partial_\tau\rho}{\varepsilon} \ \text{if} \ z=0,1.
	\end{gather*}
	The Faraday law  $\partial_{\tau}\h=-\frac{1}{Z}\nabla \times \e $ yields the following Dirichlet-type boundary conditions:
	\begin{equation}
		\label{eq:BC_H}
		\begin{gathered}
			\partial_\tau H_x=0 \ \text{if} \ x=0,1,\\
			\partial_\tau H_y=0 \ \text{if} \ y=0,1,\\
			\partial_\tau H_z=0 \ \text{if} \ z=0,1.
		\end{gathered}
	\end{equation}
	In Table~\ref{tab:bc_terms}, we provide the values of the second-order derivatives at the boundaries of $\Omega$ that we need for subsequent analysis and that can be derived from equations (\ref{eq:BC_E}) and (\ref{eq:BC_H}) by differentiation.
	
	\begin{table}[ht!]
		\centering
		\begin{tabular}{|p{2cm}|p{2cm}|p{2cm}|p{2cm}|}
			\hline
			&$x=0,1$&$y=0,1$ & $z=0,1$	\\ [0.5ex]
			\hline
			$\partial_{\tau}^2E_x$ &-& 0& 0 \\[1mm]
			$\partial_{\tau}^2E_y$& 0 &- & 0 \\[1mm]
			$\partial_{\tau}^2E_z$&0& 0& - \\[1mm]
			$\partial_{\tau}\partial_yH_z$&-&-&0 \\[1mm]
			$\partial_{\tau}\partial_zH_y$ & - & 0 & - \\[1mm]
			$\partial_{\tau}\partial_zH_x$&0& -& - \\[1mm]
			$\partial_{\tau}\partial_xH_z$&-&-&0 \\[1mm]
			$\partial_{\tau}\partial_xH_y$&-& 0& - \\[1mm]
			$\partial_{\tau}\partial_yH_x$&0 & -& - \\[1mm]
			\hline
		\end{tabular}
		\caption{Second-order derivatives in time and mixed space-time derivatives at the boundaries of $\Omega$. The first three rows follow from (\ref{eq:BC_E}). The other six rows are an implication of (\ref{eq:BC_H}).}
		\label{tab:bc_terms}
	\end{table}

	The Amp\`ere law $\partial_{\tau}\e=Z(\nabla \times \h-\J)$ implies that
	\begin{align*}
		&
		\partial_\tau^2E_x=Z\left(
		\frac{\partial H_z}{\partial_y\partial_\tau}-\frac{\partial^2 H_y}{\partial_z\partial_\tau}-\partial_\tau \J_x
		\right),    \\&
		\partial_\tau^2E_y=Z\left(
		\frac{\partial H_x}{\partial_z\partial_\tau}-\frac{\partial^2 H_z}{\partial_x\partial_\tau}-\partial_\tau \J_y
		\right), \\&
		\partial_\tau^2E_z=Z\left(
		\frac{\partial H_y}{\partial_x\partial_\tau}-\frac{\partial^2 H_x}{\partial_y\partial_\tau}-\partial_\tau \J_z
		\right).
	\end{align*}
	Substituting the values from  Table \ref{tab:bc_terms}
	into these  equations, we obtain  the following six Neumann-type boundary conditions:
	\begin{gather*}
		\partial_y\partial_\tau H_x=-\partial_\tau \J_z \ \text{if} \ y=0,1 \quad
		\partial_z\partial_\tau H_x=\partial_\tau \J_y \ \text{if} \ z=0,1,\\
		\partial_x\partial_\tau H_y=\partial_\tau \J_z \ \text{if} \ x=0,1 \quad
		\partial_z\partial_\tau H_y=-\partial_\tau \J_x \ \text{if} \ z=0,1,\\
		\partial_x\partial_\tau H_z=-\partial_\tau \J_y \ \text{if} \ x=0,1 \quad
		\partial_y\partial_\tau H_z=\partial_\tau \J_x \ \text{if} \ y=0,1.
	\end{gather*}
	
	With these arrangements, the derivatives $\partial_{\tau}\e$ and $\partial_{\tau}\h$
	are supplied with well-defined boundary conditions, as summarized in Table \ref{tab:bc}.
	In short, on $\partial \Gw$ we have   $\partial_{\tau}{\bf E}_{\parallel} = \partial_{\tau}{\bf H}_{\perp} = 0$, while
	$\partial_{\tau}{\bf E}_{\perp}$ and $\partial_{\tau}{\bf H}_{\parallel}$ obey a Neumann condition.
	As discussed earlier, Dirichlet conditions are true boundary conditions while Neumann conditions
	are derived from the interior equations.
	
	\begin{table} [ht!]
		\centering
		\begin{tabular}{|p{2cm}|p{2cm}|p{2cm}|p{2cm}|}
			\hline
			&$x=0,1$&$y=0,1$ & $z=0,1$	\\ [0.5ex]
			\hline
			$\partial_{\tau}E_x$&Neumann&Dirichlet&Dirichlet \\
			$\partial_{\tau}E_y$&Dirichlet &Neumann&Dirichlet \\
			$\partial_{\tau}E_z$&Dirichlet&Dirichlet&Neumann \\
			$\partial_{\tau}H_x$&Dirichlet&Neumann&Neumann \\
			$\partial_{\tau}H_y$&Neumann&Dirichlet&Neumann \\
			$\partial_{\tau}H_z$&Neumann&Neumann&Dirichlet \\
			\hline
		\end{tabular}
		\caption{Boundary conditions for the time derivatives of the Cartesian components of  $\e$ and $\h$.}
		\label{tab:bc}
	\end{table}
	
	\section{The scheme}\label{sec:the scheme}
	\subsection{Equation based derivation}
	
	We now extend the second-order accurate Yee scheme to a fourth-order accurate scheme, in both space and time, while maintaining the compact Yee stencil.
	The idea is to use a Taylor  expansion in time to the next order and then replace the resulting third-order time derivatives by space derivatives using
	Maxwell's equations.
	
	The fourth-order Taylor expansion applied to \eqref{eq:maxwellv} and combined with equations \eqref{secondorder} and \eqref{secondorderh} yields:
	{\allowdisplaybreaks
		\begin{align*}
			&
			Z(\nabla \times \h^{n+1/2}-\J^{n+1/2})=	\left (  \frac{\partial \e}{\partial \tau} \right)^{n+1/2}=
			\frac{\e^{n+1} -\e^n}{h_{\tau}}-\frac{h_{\tau}^2}{24}\frac{\partial^3 \e^{n+1/2}}{\partial \tau^3}+O(h_{\tau}^4)=\\&
			\delta_{\tau}\e^{n+1/2}-\frac{h_{\tau}^2}{24}\partial_{\tau}
			\left(
			-Z\partial_\tau \J^{n+1/2}-\nabla(\nabla \cdot \e^{n+1/2})+\Delta \e^{n+1/2}
			\right)+O(h_{\tau}^4)=\\&
			\delta_{\tau}\e^{n+1/2}-\frac{h_{\tau}^2}{24}
			\left(
			-Z\partial_{\tau}^2\J^{n+1/2}
			-\nabla(\nabla\cdot\partial_\tau \e^{n+1/2})+\Delta \partial_{\tau}\e^{n+1/2}
			\right)+O(h_{\tau}^4)=\\&
			\delta_{\tau}\e^{n+1/2}-\frac{h_{\tau}^2}{24}
			\left(
			-Z\partial_{\tau}^2\J^{n+1/2}
			+Z\nabla(\nabla\cdot\J^{n+1/2})+\Delta \partial_{\tau}\e^{n+1/2}
			\right)+O(h_{\tau}^4).
		\end{align*}
	}%
	Since  $\Delta \partial_{\tau}\e^{n+1/2}=\Delta\delta_{\tau} \e^{n+1/2}+O(h_{\tau}^2) $,
	we have that
	\begin{align*}
		&
		Z(\nabla \times \h^{n+1/2}-\J^{n+1/2})=\\&
		\delta_{\tau}\e^{n+1/2}-\frac{h_{\tau}^2}{24}
		\left(
		-Z\partial_{\tau}^2\J^{n+1/2}
		+Z\nabla(\nabla\cdot\J^{n+1/2})+\Delta \delta_{\tau}\e^{n+1/2}
		\right)+O(h_{\tau}^4).
	\end{align*}
	Similarly, we have for $\h$:
	\begin{align*}
		&
		-\frac{1}{Z}\nabla \times \e^{n+1}=\frac{\partial \h^{n+1}}{\partial \tau}=\\&
		\delta_\tau \h^{n+1}-\frac{h_{\tau}^2}{24}\partial_{\tau}\left (
		\Delta \h^{n+1}-\nabla \times \partial_{\tau}\J^{n+1}
		\right)+O(h_{\tau}^4)=\\&
		\delta_\tau \h^{n+1}-\frac{h_{\tau}^2}{24}\left (
		\delta_{\tau}\Delta\h^{n+1}-\nabla \times \partial_{\tau}^2\J^{n+1}
		\right)+O(h_{\tau}^4).
	\end{align*}
	Thus, at each time step we obtain several uncoupled modified Helmholtz equations which are positive definite.
	Note, that for more complicated situations the equations might be coupled at the boundary.
	\begin{subequations}
		\label{eq:mh}
		\begin{align}
			-\Delta\delta_{\tau} \e^{n+\frac{1}{2}}+\frac{24}{h_{\tau}^2}\delta_{\tau}\e^{n+\frac{1}{2}} &=
			Z\frac{24}{h_{\tau}^2}\nabla\times \h^{n+\frac{1}{2}}+P({\bf J})^{n+\frac{1}{2}}+
			O(h_{\tau}^4+h^4) \quad \mathrm{in} \quad (0,1)^3, \label{eq:mhe} \\
			-\Delta \delta_{\tau}\h^{n+1}+\frac{24}{h_{\tau}^2}\delta_{\tau}\h^{n+1} &=
			-\frac{1}{Z}\frac{24}{h_{\tau}^2}\nabla\times \e^{n+1}
			-\nabla\times\partial_{\tau}^2\J^{n+1}+
			O(h_{\tau}^4+h^4)
			\quad \mathrm {in} \quad (0,1)^3,\label{eq:mhh}
		\end{align}
	\end{subequations}
	where
	\begin{equation}
		\label{eq:pj}
		P({\bf J})^{n+\frac{1}{2}}\stackrel{\text{def}}{=}Z\left ( -{\bf J}^{n+\frac{1}{2}}-\nabla (\nabla \cdot{\bf J}^{n+\frac{1}{2}})+\frac{\partial^2{\bf J}^{n+\frac{1}{2}}}{\partial \tau^2}\right ).
	\end{equation}
	Equations (\ref{eq:mh}) are supplemented by the boundary conditions for the functions $\delta_{\tau}\e$ and  $\delta_{\tau}\h$, as discussed in Section~\ref{sec:prelimi}.

	\subsection{Modified Helmholtz equation}\label{subsec:mh}

	In a  Cartesian coordinate system, each of the equations \eqref{eq:mhe} and \eqref{eq:mhh} gets split into three scalar modified Helmholtz equations:
	\begin{equation}\label{helmholtz}
		-\Delta \phi+\kappa^2\phi=\kappa^2F, \ \kappa^2>0.
	\end{equation}
	We use the  fourth-order accurate compact finite-difference scheme \cite{singer_turkel} for solving  elliptic equations of the type (\ref{helmholtz}).
	
	Consider an equally-spaced mesh of dimension $N_x\times N_y\times N_z$ and size $h$ in each direction.
	We denote by $D_{xx}$, $D_{yy}$, and $D_{zz}$  the standard  second-order central-difference operators  and define
	$$
	\Delta_h\stackrel{\text{def}}{=}D_{xx}+D_{yy}+D_{zz},
	\quad \varUpsilon_h\stackrel{\text{def}}{=} D_{zz}D_{xx}+D_{yy}D_{zz}+D_{xx}D_{yy}.
	$$
	Then, the fourth-order accurate scheme for \eqref{helmholtz} is given by \cite{singer_turkel}:
	\begin{equation}\label{eq:ell}
		-\left (\Delta_h+\frac{h^2}{6}\varUpsilon_h\right)
		\phi+\kappa^2\left (1+\frac{\kappa^2h^2}{12}\right)\phi=
		\kappa^2\left (1+\frac{\kappa^2h^2}{12}+\frac{h^2}{12}\Delta_h \right)F.
	\end{equation}
	If $\Delta F$ is  known at the grid nodes to fourth-order accuracy, then
	\eqref{helmholtz} can be simplified further to
	\begin{equation}\label{eq:helm2}
		-\left (\Delta_h+\frac{h^2}{6}\varUpsilon_h\right)
		\phi+\kappa^2\left (1+\frac{\kappa^2h^2}{12}\right)\phi=
		\kappa^2\left (1+\frac{\kappa^2h^2}{12} \right)F+
		\frac{\kappa^2h^2}{12}\Delta F.
	\end{equation}
	
	Letting $\kappa^2 \!=\! \dfrac{24}{h_{\tau}^2}$ in  \eqref{eq:mh} and \eqref{eq:ell} gives rise to  six elliptic equations of type (\ref{helmholtz}) with $\phi$ and $F$ specified in  Table \ref{tab:spec}. Once the discretization in space has been applied (Section~\ref{sec:discrete}), equations \eqref{eq:mh} lead to six positive definite linear systems to be solved independently by  conjugate gradients.
	Note that, it is difficult to use multigrid since the grid dimension for different various equations is not the same and hence not all equations can have $2^n$ grid nodes in every direction.

	\begin{table} [ht!]
		\centering
		\begin{tabular}{|p{0.5cm}|p{2cm}|p{8cm}|}
			\hline
			Eq.  &$\phi$&F	\\ [0.5ex]
			\hline
			\label{tab:6eq}
			1&$\delta_{\tau}E_x^{n+\hf}$&
			$Z(\partial_y H_z^{n+\hf}-\partial_z H_y^{n+\hf})+\frac{h_{\tau}^2}{24}P(\J^{n+\hf})_x$
			\\
			2&$\delta_{\tau}E_y^{n+\hf}$&
			$Z(\partial_z H_x^{n+\hf}-\partial_x H_z^{n+\hf})+\frac{h_{\tau}^2}{24}P(\J^{n+\hf})_y$
			\\
			3&$\delta_{\tau}E_z^{n+\hf}$&
			$Z(\partial_x H_y^{n+\hf}-\partial_y H_x^{n+\hf})+\frac{h_{\tau}^2}{24}P(\J^{n+\hf})_z$
			\\
			4&$\delta_{\tau}H_x^{n+\frac{3}{2}}$&
			$-\frac{1}{Z}(\partial_y E_z^{n+1}-\partial_z E_y^{n+1})-\frac{h_{\tau}^2}{24}(\nabla \times \partial_{\tau}^2\J^{n+1})_x$
			\\
			5&$\delta_{\tau}H_y^{n+\frac{3}{2}}$&
			$-\frac{1}{Z}(\partial_z E_x^{n+1}-\partial_x E_z^{n+1})-\frac{h_{\tau}^2}{24}(\nabla \times \partial_{\tau}^2\J^{n+1})_y$
			\\
			6&$\delta_{\tau}H_z^{n+\frac{3}{2}}$&
			$-\frac{1}{Z}(\partial_x E_y^{n+1}-\partial_y E_x^{n+1})-\frac{h_{\tau}^2}{24}(\nabla \times \partial_{\tau}^2\J^{n+1})_z$
			\\
			\hline
		\end{tabular}
		\caption{The expressions for  $\phi$ and $F$ in (\ref{helmholtz})
			($P(\J)$ is given by \eqref{eq:pj}).}
		\label{tab:spec}
	\end{table}

	\subsection{Neumann boundary conditions}
	
	We reemphasize that Neumann conditions for a first-order system are not imposed boundary conditions. They are rather an implication of the PDEs in the interior.
	Let $\Gw=[0,1]^3$.
	We define ghost nodes outside the numerical grid so
	that the Neumann conditions are satisfied to fourth-order accuracy (see Figure \ref{fig:ghost}).

	\begin{figure}[ht!]
		\centering
			\includegraphics[scale=1]{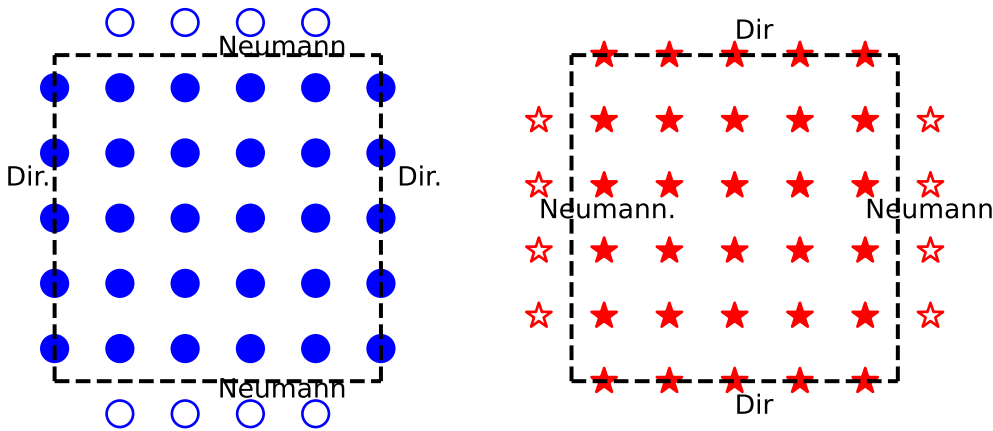}
			\caption{
				Left: The grid nodes for $H_x$ projected onto the $(x,y)$ plane in the case $N_x=N_y=5$. The unit square is bounded by the dashed line. Solid circles show the mesh points of $\Gw_h^{H_x}$. Hollow circles denote the ghost points induced by the Neumann boundary conditions at $y=0,1$.
				Right: The grid nodes for $H_y$ projected onto the $(x,y)$ plane in the case $N_x=N_y=5$. The unit square is bounded by the dashed frame. Solid stars show the mesh points of $\Gw_h^{H_y}$. Hollow stars denote ghost points induced by the Neumann boundary conditions at $x=0,1$.}
			\label{fig:ghost}
		\end{figure}

		We propose two methods for the implementation of the Neumann boundary conditions.
		\paragraph{\em Method 1 (equation based)}
		Consider, without loss of generality, equation \eqref{helmholtz} with the Neumann boundary condition $\partial_x \phi(0,y,z)=g(y,z)$.
		Differentiating \eqref{helmholtz} with respect to $x$ (cf. \cite{singer_turkel}), we have:
		\begin{equation}\label{eq:gst}
			\phi_x(0,x,z)=\frac{\phi(\hf h,y,z)-\phi(-\hf h,y,z)}{h}+\frac{h^2}{24}
			\left(\phi_{xyy}+\phi_{xzz}-\kappa^2\phi_x-\kappa^2F_x \right)+O(h^4).
		\end{equation}
		Assume that $\phi_{xyy}+\phi_{xzz}$ is known for $x=0$.
		Setting $x\!=\!0$ in \eqref{eq:gst}, we define the ghost value $\phi(-\hf h,y,z)$ such that the Neumann boundary condition $\partial_x \phi\!=\!=g(y,z)$ is satisfied to fourth-order.
		\\[2mm]
		\paragraph{\em Method 2 (Taylor based)}
		By Taylor's expansion
		\begin{equation}\label{eq:gst2}
			\phi_x(0,y,z)=\frac{\phi(\hf h,y,z)-\phi(-\hf h,y,z)}{h}-\frac{h^2}{24}\partial_{xxx}\phi(0,y,z)+O(h^4).
		\end{equation}
		Hence, knowing $\partial_x\phi$ and $\partial_{xxx}\phi$ at $x\!=\!0$ enables one to obtain the ghost variable $\phi(-\hf h,y,z)$ using \eqref{eq:gst2}.
		
		Next, we  show how to approximate the Neumann-type boundary conditions to fourth-order accuracy for $\partial_{\tau}\e$ using Method 1 and
		for $\partial_{\tau}\h$ using Method 2.
		\begin{lemma}
			Let $\phi(x,y,z)=\partial_\tau E_x$ be evaluated at some fixed $\tau_0$.
			Then, the quantities $$\phi_{xyy}(0,y,z),\  \phi_{xzz}(0,y,z),\  \phi_x(0,y,z),\ \text{and}\ F_x$$ are known to fourth-order accuracy.
			(Recall that, $F$ contains the terms involving only  $\partial_zH_y-\partial_y H_z$ and $\J$.)
			If, in addition, $\J=0$ and $\rho=0$, then
			$\phi_{xyy}(0,y,z)=\phi_{xzz}(0,y,z)=\phi_x(0,y,z)=F_x(0,y,z)=0$.
		\end{lemma}
		\begin{proof}
			$\phi_x(0,y,z)$ equals to $\partial_\tau E_x$ at $x=0$, which  has already been specified.
			Moreover, we have
			\begin{equation}\label{eq:lem_eq}
				\partial_x E_x+\partial_y E_y+\partial_z E_z=\frac{\rho}{\varepsilon}.
			\end{equation}
			Differentiating \eqref{eq:lem_eq} twice with respect to $y$ and once with respect to $\tau$, we arrive at
			$$
			-\partial_{xyy}\phi+\partial_{\tau yy}(\rho/\varepsilon)=
			\partial_{yyy}\phi+\partial_{zyy}\phi.
			$$
			If $x=0$, then $\partial_{\tau yyy}E_x=\partial_{\tau zyy}E_x=0$,
			and hence if $x=0$, then $\partial_{xyy}\phi=\partial_{\tau yy}(\rho/\varepsilon)$.
			Differentiating \eqref{eq:lem_eq} twice with respect to $z$ and once with respect to $\tau$, we get
			$$
			-\partial_{xzz}\phi+\partial_{\tau zz}(\rho/\varepsilon)=
			\partial_{yzz}\phi+\partial_{zzz}\phi.
			$$
			If $x=0$, then $\partial_{\tau zzz}E_x=\partial_{\tau yzz}E_x=0$,
			and hence at $x=0$ we have $\partial_{xzz}\phi=\partial_{\tau zz}(\rho/\varepsilon)$.
			
			To obtain $F_x$ at $x=0$, it is sufficient to evaluate
			$\partial_x (\partial_z H_y- \partial_y H_z  )$.
			This quantity, in turn, can be derived via the Amp\`ere law $\partial_\tau \e=Z(\nabla \times \h-\J)$ if 
			$\partial_{x\tau}E_x$ is available. And the latter is known through the Neumann boundary condition for $\partial_{\tau}\e$ at $x=0$.
			
			The second assertion follows immediately by virtue of the same argument.
		\end{proof}
		
		By repeating the previous proof for  the remaining faces of the cube $\Omega=[0,\,1]^3$ and  field components $E_y$ and $E_z$, we obtain the following 
		
		\begin{cor}
			Neumann boundary conditions for equation \eqref{eq:mhe} admit a compact fourth-order accurate discretization by means of Method 1.
		\end{cor}

		Next, we will show how one can use Method 2 to build a compact discretization of  the Neumann boundary conditions for $\partial_\tau \h$.
		
		\begin{lemma}
			Let $\phi(x,y,z)=\partial_\tau H_y$ be evaluated at some fixed $\tau_0$. Then, the term $\phi_{xxx}$ at $x=0,1$ can be approximated to fourth-order accuracy.
			
		\end{lemma}
		\begin{proof}
			Assume without loss of generality that $Z=1$.
			The equation 
			\begin{equation}\label{eq:lem2}
				\partial_\tau E_z=\partial_{x}H_y-\partial_y H_x-\J_z
			\end{equation}
			(cf.\ (\ref{maxe})) and  boundary conditions $\partial_{\tau}E_z=0$ at $x=0,1$
			imply that at $x=0,1$
			$$
			\partial_{\tau xxx}H_y=\partial_{\tau yxx}H_x+\partial_{\tau xx}\J_z.
			$$
			Then, we use $\partial_\tau\nabla \cdot \h=0$  and replace
			$\partial_x H_x$ with $-\partial_{\tau y} H_y-\partial_{\tau z} H_z$.
			This yields:
			$$
			\partial_{\tau yxx}H_x=-\partial_{\tau yyx}H_y-\partial_{\tau xyz}H_z.
			$$
			To complete the proof, we have to approximate
			$\partial_{\tau yyx}H_y$ and $\partial_{\tau xyz}H_z$. This is done as follows.
			We use \eqref{eq:lem2} again to obtain 
			$$
			\partial_{\tau yyx}H_y=\partial_{\tau yyy}H_x+\partial_{\tau yy}\J_z.
			$$
			At $x=0,1$, the first equation of (\ref{eq:BC_H}) implies that $\partial_{\tau y y y}H_x=0$.
			
			To approximate $\partial_{\tau x y z}H_z$, we use   the $y$-component of the Amp\`ere law (\ref{maxe}):
			$$
			\partial_\tau E_y=\partial_{z}H_x-\partial_{x}H_z-\J_y.
			$$
			Therefore, at $x=0,1$ we have:
			$$
			\partial_{x}H_z=\J_y+\partial_z H_x
			$$
			and
			$$\partial_{\tau x y z}H_z=\partial_{\tau yz }\J_y+
			\partial_{\tau yzz }H_x.
			$$
			Since $\partial_{\tau}H_x=0$ at $x=0,1$, we finally derive:
			$$
			\partial_{\tau y zz}H_x=0.
			$$
			This completes the proof.
		\end{proof}
		By repeating the previous procedure for $H_x$ and $H_z$, we obtain the following corollaries.
		\begin{cor}
			Neumann boundary conditions for equation \eqref{eq:mhh} admit a compact fourth-order accurate discretization by means of Method 2.
		\end{cor}
		\begin{cor}\label{cor:ghost_cor}
			Assume that $\rho=0$ and $\J=0$.
			Then, all Neumann boundary conditions obtained for $\partial_\tau\e$ and $\partial_\tau \h$ are homogeneous
			and all the  derivatives in \eqref{eq:gst}, \eqref{eq:gst2} vanish. In particular,
			$\phi(\tilde{x}-\hf)=\phi(\tilde{x}+\hf)$ is a fourth-order accurate approximation to the homogeneous Neumann condition
			at the boundary point $\tilde{x}$ (see  also Section \ref{sec:TE}).
		\end{cor}

		We emphasize that, the right hand side, $F$, needs to be approximated with fourth-order accuracy.
		This is done using a fourth-order Pad\'{e} approximation for the curl operator (see details in Appendix \ref{subsec:op_details}).
		Then, we solve equations 1, 2, and 3 from Table~\ref{tab:spec} using the scheme \eqref{eq:ell}, and subsequently solve equations 4, 5, and 6 from Table~\ref{tab:spec} using the scheme \eqref{eq:helm2}.
		
		In view of equations \eqref{eq:ell} and \eqref{eq:helm2}, we also define the following operators.
		\begin{defi}
			\label{def:operatorsP}
			For any $s=x,y,z$ and $G= H_s, E_s,$ 
			let $$P_{1}^{G}, P_{2}^{G}:\intr\Gwh^{G}\to \intr\Gwh^{G}$$ be the symmetric operators given by
			$$
			P_{1}^{G}:=-\left (\Delta_h+\frac{h^2}{6}\varUpsilon_h\right)+\frac{24}{h_{\tau}^2}\left (
			1+\frac{2}{r^2}
			\right),
			$$
			$$
			P_{2}^{G}:=\frac{24}{h_{\tau}^2}\left (
			-1-\frac{2}{r^2}-\frac{h^2}{12}\Delta_h
			\right).
			$$
			We omit the superscript $G$ whenever it does not lead to any ambiguity.
		\end{defi}
		
		Hereafter,  $r=\frac{h_{\tau}}{h}$ will denote the CFL number.
		As shown in  Appendix \ref{app:operatorestimates},  if $r<\sqrt{3+\sqrt{21}}$, then
		$P_2^G$ and  $P_1^G$ are symmetric  positive definite matrices and
		the following estimates hold:
		\begin{align}\label{eq:P1P2estimates}
			&
			0<\frac{2}{r^2}\leq \|\frac{h_{\tau}^2}{24}P_2^G\|\leq 1+\frac{2}{r^2},\\&
			1<1+\frac{2}{r^2}-\frac{r^2}{6}\leq\|\frac{h_{\tau}^2}{24}P_1^G\|\leq 1+\frac{2}{r^2}+\frac{r^2}{2}.
		\end{align}
		In particular,
		$$
		\|\left (P_{1}^{G}\right )^{-1}P_{2}^{G}\|\geq \frac{\frac{2}{r^2}}{1+\frac{2}{r^2}+\frac{r^2}{2}}.
		$$

		\subsection{The numerical scheme}\label{sec:discrete}
		Let
		$$
		\curl_h\stackrel{\text{def}}{=}
		\begin{pmatrix}
			0& -\delta_z & \delta _y\\
			\delta_z&0&-\delta_x\\
			-\delta_y&\delta_x&0\\
		\end{pmatrix}
		$$
		be the block matrix  representing a spatial Pad\'{e} fourth-order finite-difference approximation of the curl operator.
		See Appendix \ref{subsec:op_details} for the full definition of the matrix $\curl_h$.
		
		Let $\e^n_h$ and $\Delta\e^n_h $  be given on $\Gwh^{\e}$ and let $\h^{n+\frac{1}{2}}_h$ be given on $\Gwh^{\h}$. To advance in time, we take two steps.\\[2mm]
		{\bf Step 1}:
		$$
		{\e}_h^{n+1}=
		\begin{pmatrix}
			E_x^n+h_{\tau} E_x^{*}\\
			E_y^n+h_{\tau} E_y^{*}\\
			E_y^n+h_{\tau} E_y^{*}
		\end{pmatrix}
		$$
		where
		\begin{align*}
			&
			\underbrace{
				\begin{pmatrix}
					P_1^{E_x} & 0&0 \\
					0 & P_1^{E_y} &\\
					0&0&P_1^{E_z}
				\end{pmatrix}
			}_{P_1^\e}
			\begin{pmatrix}
				E_x^{*}\\
				E_y^{*} \\
				E_z^{*}
			\end{pmatrix}=ZP(\J^{n+1/2})+\\&\kappa^2Z
			\begin{pmatrix}
				0& -\delta_z & \delta _y\\
				\delta_z&0&-\delta_x\\
				-\delta_y&\delta_x&0
			\end{pmatrix}
			\left [
			\left(1+\frac{\kappa^2 h^2}{12} \right)
			\begin{pmatrix}
				H_x^{n+\frac{1}{2}}\\
				H_y^{n+\frac{1}{2}}\\
				H_z^{n+\frac{1}{2}}
			\end{pmatrix}+
			\frac{h^2}{12}
			\begin{pmatrix}
				\Delta_h H_x^{n+\frac{1}{2}}\\
				\Delta_h H_y^{n+\frac{1}{2}}\\
				\Delta_h H_z^{n+\frac{1}{2}}
			\end{pmatrix}
			\right]
		\end{align*}
		(cf.\ Section \ref{subsec:mh}).
		Extend ${\e}_h^{n+1}$ to $\Gwh^\e$ using the boundary conditions.
		Define
		$$\Delta \e_h^{n+1} =\Delta \e_h^n+\kappa^2(\e_h^{n+1}-\e_h^n)-Z\frac{24}{h_{\tau}}
		\begin{pmatrix}
			0& -\delta_z & \delta _y\\
			\delta_z&0&-\delta_x\\
			-\delta_y&\delta_x&0
		\end{pmatrix}
		\begin{pmatrix}
			H_x^{n+\frac{1}{2}}\\
			H_y^{n+\frac{1}{2}}\\
			H_z^{n+\frac{1}{2}}
		\end{pmatrix}-h_{\tau} P(\J)^{n+\hf}
		$$
		on $\Gwh^\e$
		(see \eqref{eq:mh})
		\\[2mm]
		{\bf Step 2}:
		$$
		{\h}_h^{n+\frac{3}{2}}=
		\begin{pmatrix}
			H_x^{n+\hf}+h_{\tau} H_x^{*}\\
			H_y^{n+\hf}+h_{\tau} H_y^{*}\\
			H_y^{n+\hf}+h_{\tau} H_y^{*}
		\end{pmatrix}
		$$
		where
		\begin{align*}
			&
			\underbrace{
				\begin{pmatrix}
					P_1^{H_x}  & 0&0 \\
					0 & 	P_1^{H_y}  &\\
					0&0&	P_1^{H_z}
				\end{pmatrix}
			}_{P_1^\h}
			\begin{pmatrix}
				H_x^{*}\\
				H_y^{*} \\
				H_z^{*}
			\end{pmatrix}-\frac{1}{Z}\nabla \times \partial_{\tau}^2\J^{n+1}
			=\\&
			-\frac{\kappa^2}{Z}
			\begin{pmatrix}
				0& -\delta_z & \delta _y\\
				\delta_z&0&-\delta_x\\
				-\delta_y&\delta_x&0
			\end{pmatrix}
			\left [
			\left(1+\frac{\kappa^2h^2}{12} \right)
			\begin{pmatrix}
				E_x^{n+1}\\
				E_y^{n+1}\\
				E_z^{n+1}
			\end{pmatrix}+
			\frac{h^2}{12}
			\begin{pmatrix}
				\Delta E_x^{n+1}\\
				\Delta E_y^{n+1}\\
				\Delta E_z^{n+1}
			\end{pmatrix}
			\right],
		\end{align*}
		(cf.\ Section \ref{subsec:mh}).
		Extend $\h_h^{n+3/2}$
		to $\Gwh^\h$ using the appropriate boundary condition.

		\section{stability analysis}\label{sec:stability}
		We assume, with no loss of generality, that $\rho\!=\!0$, $\J \!=\! 0$, and $Z \!=\! 1$.
		The scheme can therefore be written in a compact way:
		\begin{subequations}
			\label{eq:eh}
			\begin{align}
				\e_h^{n+1} -\e_h^{n} &= h_{\tau} P_1^{-1} \curl_h
				\left(
				\kappa^2\left(1+\frac{\kappa^2h^2}{12}\right)\h_h^{n+1/2}+\frac{\kappa^2h^2}{12}\Delta \h_h^{n+1/2}\right),
				\label{eq:c1} \\
				\h_h^{n+3/2} -\h_h^{n+1/2} &= -h_{\tau} P_1^{-1} \curl_h
				\left(
				\kappa^2\left(1+\frac{\kappa^2h^2}{12}\right) \e_h^{n+1}+\frac{\kappa^2h^2}{12}\Delta \e_h^{n+1} \right).
				\label{eq:c2}
			\end{align}
		\end{subequations}
		We  approximate $\Delta $  in equations \eqref{eq:c1} and \eqref{eq:c2}  using the standard second-order difference  Laplacian $\Delta_h$ (same  as in \eqref{eq:ell}).
		Then, we recast  \eqref{eq:c1}, \eqref{eq:c2} as:
		\begin{align}
			\nonumber
			&\left (\bf{1}-
			h_{\tau}
			\begin{pmatrix}
				0&0\\
				-P_1^{-1}P_2\curl_h& 0
			\end{pmatrix}
			\right)
			\begin{pmatrix}
				\e^{n+1} \\
				\h^{n+3/2}
			\end{pmatrix}=  \\
			\label{eq:stab1}
			&\left (\bf{1}-h_{\tau} 
			\begin{pmatrix}
				0 &-P_1^{-1}P_2\curl_h\\
				0& 0
			\end{pmatrix}
			\right)
			\begin{pmatrix}
				\e^{n} \\
				\h^{n+1/2}
			\end{pmatrix},
		\end{align}
		where the operators $P_1$ and $P_2$ are introduced in Definition~\ref{def:operatorsP}.
		Next, assume  the solution is in the form of a plane wave:
		$$
		\begin{pmatrix}
			\e^{n}\\
			\h^{n+1/2}
		\end{pmatrix}_{i,j,k}=
		\sigma^n
		(\exp(-\sqrt{-1}(\xi_xih+\xi_yjh+\xi_zkh)))
		\begin{pmatrix}
			\e_0\\
			\h_0
		\end{pmatrix}.
		$$
		Let  $Q=-P_1^{-1} P_2 \curl_h$ and
		let $\lambda$ denote an eigenvalue of $Q$.
		We substitute the plane wave solution into \eqref{eq:stab1} and derive
		$$
		\begin{pmatrix}
			\sigma-1& \lambda h_{\tau}\\
			-\sigma h_{\tau} \lambda & \sigma-1
		\end{pmatrix}
		\begin{pmatrix}
			\e_0\\
			\h_0
		\end{pmatrix}=
		\begin{pmatrix}
			0\\
			0
		\end{pmatrix}
		.
		$$
		Therefore,
		$$
		\sigma^2-\sigma(2-\lambda^2h_{\tau}^2)+1=0.
		$$
		In particular, $|h_{\tau}\lambda|\leq 2$ implies that $|\sigma|\leq 1$. Then, using the definition of the matrix $A$ given by equation \eqref{eq:lhs} and Remark \ref{rem:curl estimates} (see Appendix \ref{subsec:op_details}), we can derive the stability  condition in the following form:
		\begin{equation}\label{eq:b_effects}
			r2\sqrt{3}\|A^{-1}\| \|-P_1^{-1}P_2\|\leq 2 \Leftrightarrow
			r\leq \frac{1}{\sqrt{3}\|A^{-1}\|\cdot \|-P_1^{-1}P_2\|}.
		\end{equation}
		By \eqref{eq:P1P2estimates}, $r<\sqrt{3+\sqrt{21}}$ implies that $\|-P_1^{-1}P_2\|\geq \frac{\frac{2}{r^2}}{1+\frac{2}{r^2}+\frac{r^2}{2}}$ (see Appendix \ref{app:operatorestimates} for detail).
		Therefore, the scheme is stable provided that
		$$
		r\leq\min\left \{
		\sqrt{3+\sqrt{21}}, \frac{1+\frac{2}{r^2}+\frac{r^2}{2}}{\sqrt{3}\|A^{-1}\|\frac{2}{r^2}}
		\right \}.
		$$
		The inequality
		$$
		\frac{1+\frac{2}{r^2}+\frac{r^2}{2}}{\frac{2}{r^2}}\geq 1, \quad r\in (0,\infty)
		$$
		gives rise to 
		a sufficient stability condition
		\begin{equation}
			0<r\leq \frac{1}{\sqrt{3}\|A^{-1}\|}\sim\frac{5}{6\sqrt{3}}.
		\end{equation}

		In the following section we will examine a two-dimensional reduction of Maxwell's equation. A modification of  Remark \ref{rem:curl estimates} readily implies the stability condition 
		\begin{equation}
			\label{stabCFL}
			0<r\leq \frac{1}{\sqrt{3}\|A^{-1}\|}\sim\frac{5}{6\sqrt{2}}
		\end{equation}
		(cf. \cite{yefet_turkel}).
		The numerical results summarized in Table \ref{table:cfl} corroborate the stability estimate (\ref{stabCFL}).
		\section{Example: Transverse magnetic waves  in $[0,1]^2$}\label{sec:TE}
		
		\subsection{Numerical simulations data}
		
		For the computations, we consider the scaled two-dimensional TM system without any current or charges.
		Thus, the equations reduce to
		\begin{subequations}
			\label{maxwell2d}
			\begin{align}
				\frac{\partial E_z}{\partial \tau} &= Z \left(-\frac{\partial H_x}{\partial y} + \frac{\partial H_y}{\partial x} \right), \\
				\frac{\partial H_x}{\partial \tau} &= - \frac{1}{Z} \frac{\partial E_z}{\partial y}, \\
				\frac{\partial H_y}{\partial \tau} &= \frac{1}{Z} \frac{\partial E_z}{\partial x}.
			\end{align}
		\end{subequations}
		Since  $\h$ is given at time moments $n+1/2$ we need to specify its initial condition at the time $t\!=\! \frac{\Delta_\tau}{2}$.
		By Taylor series,
		\begin{equation}
			\label{talorh}
			\h \left(\frac{h_{\tau}}{2}\right) \!=\! \h (0) \!+\! \left(\frac{h_{\tau}}{2}\right) \h_{\tau} (0) \!+\! \frac{1}{2} \left(\frac{h_{\tau}}{2}\right)^2 \h_{\tau \tau} (0)
			\!+\! \frac{1}{6} \left(\frac{h_{\tau}}{2}\right)^3 \h_{\tau \tau \tau}(0) \!+\! \frac{1}{24} \left(\frac{h_{\tau}}{2}\right)^4 \h_{\tau \tau \tau \tau} (0).
		\end{equation}
		$\h_{\tau} (0)$ is given by \eqref{maxh} and $\h_{\tau \tau} (0)$ is given by \eqref{secondorderh}. Differentiating \eqref{secondorderh} yields:
		\begin{equation*}
			\frac{ \partial^3 \h}{\partial \tau^3}=-\nabla \times \nabla \times\h_{\tau}-\nabla \times \partial_{\tau \tau}\J=
			\Delta \h_{\tau}-\nabla \times \partial_{\tau \tau}\J,
		\end{equation*}
		where, again, $\h_{\tau}$ is given by \eqref{maxh}. Similarly,
		\begin{equation*}
			\frac{ \partial^4 \h}{\partial \tau^4}=-\nabla \times \nabla \times\h_{\tau \tau}-\nabla \times \partial_{\tau^3}\J=
			\Delta \h_{\tau \tau}-\nabla \times \partial_{\tau \tau \tau}\J,
		\end{equation*}
		where $\h_{\tau \tau}$ is given by \eqref{secondorderh}.
		At the nodes on or next to a boundary, we have either Dirichlet or Neumann conditions as given in Table \ref{tab:bc}.
		For \eqref{maxwell2d}, we have assumed $\J \!=\! 0$. For the computations, we further assume $Z \!=\! 1$.
		
		Let $k_x,k_y\in \N$  and let $\omega:=Z\pi\sqrt{k_x^2+k_y^2}$.
		We consider the analytical solutions
		\begin{subequations}
			\begin{align}\label{eq:data}
				E_z(x,y)&=\cos(\omega \tau)
				\sin{\left(Z \pi  k_x x\right) } \sin{\left(Z \pi  k_y y\right) },
				\\
				H_x(x,y)&=-\frac{\sin( \omega \tau)}{\omega}
				\pi k_y \sin(Z \pi  k_x x)
				\cos{\left(Z \pi  k_y y\right) }, \\
				H_y(x,y)&=\frac{\sin\left( \omega \tau\right) }{\omega} \pi  k_x \cos{\left(Z \pi  k_x x\right) } \sin{\left( Z\pi  k_y y\right) }.
			\end{align}
		\end{subequations}


		We use the mean absolute error
		\begin{align*}
			\frac{1}{3N_{\tau}\cdot N^2}\sum_{i,j,n=0}^{N,N,N_{\tau}}
			&\left | {E_z}^n_{\mathrm{numer.}}(i,j)-{E_z}_{\mathrm{true}}^n(i,j) \right|+
			\left | {H_x}^n_{\mathrm{numer.}}(i,j)-{H_x}_{\mathrm{true}}^n(i,j) \right|+\\&
			\left | {H_y}^n_{\mathrm{numer.}}(i,j)-{H_y}_{\mathrm{true}}^n(i,j) \right|,
		\end{align*}
		where the subscripts ``$\mathrm{numer.}$'' and ``$\mathrm{true}$'' refer to the   numerical and  analytical solution, respectively.
		The ghost nodes are  easily defined using Corollary \ref{cor:ghost_cor}.
		The ghost points for $H_x$  (see Figure \ref{fig:ghost}) are defined similarly.

		\subsection{Comparable schemes:  $C4, NC, AI^h$}\label{subsec:comparble}
		We compare our proposed scheme, denoted by $C4$, with the following two schemes, which are second-order accurate  in time.
		These schemes are obtained according to the Yee updating rules:
		\begin{align*}
			&
			E_z^{n+1}=E_z^n+h_{\tau} (D_x H_y^{n+1/2}-D_yH_x^{n+1/2}),\\ &
			H_x^{n+3/2}=H_x^{n+1/2}-h_{\tau} D_y E_z^{n+1}, \\&
			H_y^{n+3/2}=H_y^{n+1/2}+h_{\tau} D_x E_z^{n+1},
		\end{align*}
		where $D_x$ and $D_y$ are finite difference operators that approximate the first derivative on different stencils.
		For the derivatives in the $x$ direction,  we consider a general stencil
		$$
		S=
		\frac{1}{h}
		\begin{pmatrix}
			-b & -a & a & b \\
			-d & -c& c &d \\
			-b & -a & a & b
		\end{pmatrix},
		$$
		whereas for the derivatives in the $y$ direction we use the transposed stencil $S^t$.
		In particular,
		\begin{align*}
			&
			\partial_x u_{ij}\sim \frac{c(u_{i+\hf,j}-u_{i-\hf,j})+
				d(u_{i+\frac{3}{2},j}-u_{i-\frac{3}{2},j})}{h}+\\&
			\frac{		a(u_{i+\hf,j+1}+u_{i+\hf,j-1}-u_{i-\hf,j+1}-u_{i-\hf,j-1})+
				b(u_{i+\frac{3}{2},j+1}+u_{i+\frac{3}{2},j-1}-u_{i-\frac{3}{2},j+1}-u_{i-\frac{3}{2},j-1})}{h}.
		\end{align*}
		For grid nodes near the boundary, we use the standard fourth-order accurate one-sided finite-difference approximation of the first derivatives.

		By a Taylor expansion, we obtain second-order accuracy provided that
		$$c+3d+2a+6b=1,$$
		and fourth-order accuracy if the additional constraints hold:
		\begin{subequations}
			\label{paramac}
			\begin{align}
				c+27d+2a+54b &=0, \\
				a+3b &=0.
			\end{align}
		\end{subequations}
		We define the following stencils:
		\begin{equation}
			\label{eq:K2}
			K_2=
			\frac{1}{\Delta x}
			\begin{pmatrix}
				-b & -a & a & b \\
				-d & -1+3d+2a+6b& 1-3d-2a-6b &d \\
				-b & -a & a & b
			\end{pmatrix},
		\end{equation}
		\begin{equation}
			\label{eq:K4}
			K_4=
			\frac{1}{\Delta x}
			\begin{pmatrix}
				\frac{a}{3} &-a&a&-\frac{a}{3}  \\
				-\frac{16a-1}{24}& - \frac{9-16a}{8}& \frac{9-16a}{8} &\frac{16a-1}{24} \\
				\frac{a}{3}  & - a& a & -\frac{a}{3}
			\end{pmatrix}.
		\end{equation}
		The  non-compact fourth-order accurate scheme that we denote NC exploits the stencil  $K_4$ with $a\!=\!0$.
		Its order of accuracy is $O(h^4)+O(h_{\tau}^2)$.
		
		The data-driven scheme $AI^h$ (see Appendix~\ref{subsec:AI}) uses  the stencil $K_2$, where the free parameters are obtained by means of a minimization process over the given training data.
		This scheme is of order  $O(h^2)+O(h_{\tau}^2)$.
		We use this approach since it has been recently shown to reduce the numerical dispersion for the wave equation \cite{ovadia}.
		The details of the minimization process and the selected training data are detailed in Appendix \ref{subsec:AI}.

		\begin{table} [ht!]
			\centering
			\begin{tabular}{|p{2cm}|p{3cm}|p{3cm}|p{3cm}|}
				\hline
				$N$  &NC&C4&$AI^h$
				\\ [0.5ex]
				\hline
				\multicolumn{4}{|c|}{$k_x=k_y=2$ (PPW 64), $r=\frac{5}{6\sqrt{2}}$} \\[2mm]
				\hline
				32 & 1.94 & 5.06 & 2.69 \\
				64 & 1.98 & 4.98 & 1.89 \\
				128 & 1.98 & 4.84 & 1.88 \\
				256 & 1.99 & 4.48 & 1.93 \\
				512 & 2.00 & 4.09 & 1.96 \\
				\hline
				\multicolumn{4}{|c|}{$k_x=k_y=21$
					(PPW $\sim$6), $r=\frac{5}{6\sqrt{2}}$} \\[2mm]
				\hline
				32 & 0.13 & 0.51 & 0.86 \\
				64 & 0.81 & 4.80 & 3.32 \\
				128 & 1.74 & 4.78 & 1.93 \\
				256 & 1.96 & 5.09 & 1.35 \\
				512 & 1.99 & 4.18 & 1.85 \\
				\hline
				\hline
				\multicolumn{4}{|c|}{$k_x=k_y=2$, $r=\frac{1}{6\sqrt{2}}$} \\[2mm]
				\hline
				32 & 1.31 & 5.21 & 1.88 \\
				64 & 1.86 & 4.40 & 1.97 \\
				128 & 1.97 & 3.96 & 1.99 \\
				256 & 2.00 & 3.92 & 1.99 \\
				512 & 2.00 & 3.94 & 2.00 \\
				\hline
				\multicolumn{4}{|c|}{$k_x=k_y=21$, $r=\frac{1}{6\sqrt{2}}$} \\[2mm]
				\hline
				32 & 0.02 & 0.30 & 0.43 \\
				64 & 3.17 & 4.28 & 0.55 \\
				128 & 7.46 & 4.00 & 1.61 \\
				256 & -0.97 & 3.88 & 1.94 \\
				512 & 1.66 & 3.93 & 1.99 \\
				\hline
			\end{tabular}
			\caption{Grid convergence for the three schemes we have chosen. $N$ is the grid dimension. PPW is the points-per-wavelength ratio.
				$T=\frac{1}{\sqrt{2}}$, $Z=1$. }
			\label{table:conv_rates}
		\end{table}
		
		\begin{table} [h!]
			\centering
			\begin{tabular}[p]{|c|c|c|c|}
				\hline
				$r$&NC&C4&$AI^h$
				\\ [2mm]
				\hline
				\multicolumn{4}{|c|}{$k_x=k_y=2$ (PPW 64)} \\[2mm]
				\hline
				$\frac{1}{6\sqrt{2}}$ & 2.86e-05 & 4.06e-07 & 7.92e-04 \\
				$\frac{2}{6\sqrt{2}}$ & 1.18e-04 & 3.38e-07 & 3.86e-03 \\
				$\frac{3}{6\sqrt{2}}$ & 2.66e-04 & 2.26e-07 & 3.05e-03 \\
				$\frac{4}{6\sqrt{2}}$ & 4.73e-04 & 1.01e-07 & 1.68e-03 \\
				$\frac{5}{6\sqrt{2}}$ & 7.38e-04 & 2.49e-07 & 9.23e-05 \\
				$\frac{1}{\sqrt{2}}$ &  $\infty$&$\infty$&$-$\\[2mm]
				\hline
				\multicolumn{4}{|c|}{$k_x=k_y=21$
					(PPW $\sim$6)} \\[2mm]
				\hline
				$\frac{1}{6\sqrt{2}}$ & 1.30e-01 & 5.72e-02 & 2.57e-01 \\
				$\frac{2}{6\sqrt{2}}$ & 3.39e-02 & 4.89e-02 & 2.95e-01 \\
				$\frac{3}{6\sqrt{2}}$ & 1.39e-01 & 3.49e-02 & 2.79e-01 \\
				$\frac{4}{6\sqrt{2}}$ & 2.93e-01 & 1.58e-02 & 2.55e-01 \\
				$\frac{5}{6\sqrt{2}}$ & 2.48e-01 & 2.60e-02 & 6.48e-02 \\
				$\frac{1}{\sqrt{2}}$ &  $\infty$&$\infty$& $-$\\[2mm]
				\hline
			\end{tabular}
			\caption{ The error for different CFL numbers $r$.
				$T\!=\!\frac{4}{\sqrt{2}}$, $Z=1$, $h\!=\! \frac{1}{64}$.}
			\label{table:cfl}
		\end{table}

		\subsection{Observations}
		In Table \ref{table:conv_rates}, we verify the fourth-order accuracy of our scheme C4 and provide the rates of grid convergence for the other two (NC, $AI^h$) schemes as well. 
		
		In Table \ref{table:cfl}, we examine the effect of the CFL number on the mean error and verify the results of our stability analysis (Section~\ref{sec:stability}).
		For the points-per-wavelength (PPW) ratio of 64 ($k_x=k_y=2$),
		the mean error is  smaller for scheme C4, as expected.
		Moreover, the NC scheme is more accurate than AI since it is of a higher spatial order,
		For the PPW ratio of approximately 6.4, the AI scheme shows a similar error to that of C4 even though it is of only second order. This is due to the fact that  the AI scheme was trained on coarse grids with low number of points of points per wavelength.

		In Figure \ref{fig:error(k)}, we examine the mean error as a function of the wave number.
		As expected, as the PPW ratio decreases, the C4 scheme no longer outperforms the other schemes. This is consistent with the Taylor approximation where  the local truncation error depends on higher-order derivatives that increase for shorter wavelengths (larger wavenumbers $\sqrt{k_x^2+k_y^2}$).
		
		\begin{figure}[ht!]
			\includegraphics[scale=0.6]{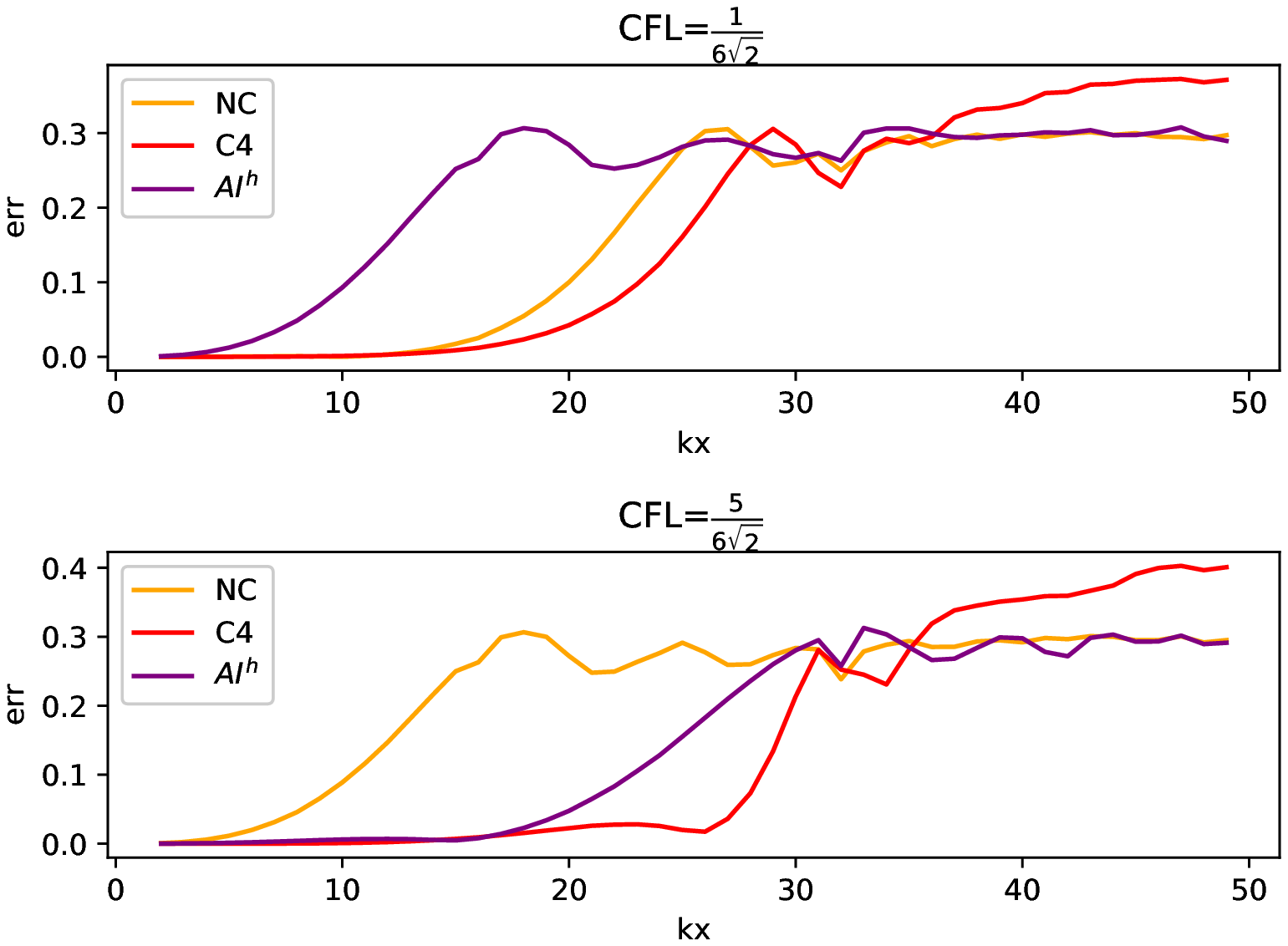}
			
			\caption{ In both panels $k_y=k_x $,
				$T=\frac{4}{\sqrt{2}}$, $Z=1$, and $h=\frac{1}{64}$. 		The points-per-wavelength ratio ranges from 64 ($\sqrt{k_x^2+k_y^2}=2$) to $\sim 2.5$ ($\sqrt{k_x^2+k_y^2}=50$).
			}
			\label{fig:error(k)}
		\end{figure}
		
		In Figures \ref{fig:6} and \ref{fig:7}, we examine how the error behaves as a function of the time step.
		As expected, scheme C4 is more accurate when the wave number is low  ($\sqrt{k_x^2+k_y^2}=2$) and the local truncation errors are relatively small.
		However, we see that for $k_x=k_y=11$ the non-compact scheme is more accurate for $r=\frac{1}{6\sqrt{2}}$
		while C4 is more accurate for $r=\frac{5}{6\sqrt{2}}$ when the time step increases and as a result the temporal error is larger.
		
		\begin{figure}[ht!]
			\begin{subfigure}[ht]{0.4\textwidth}
				\centering\includegraphics[scale=0.5]{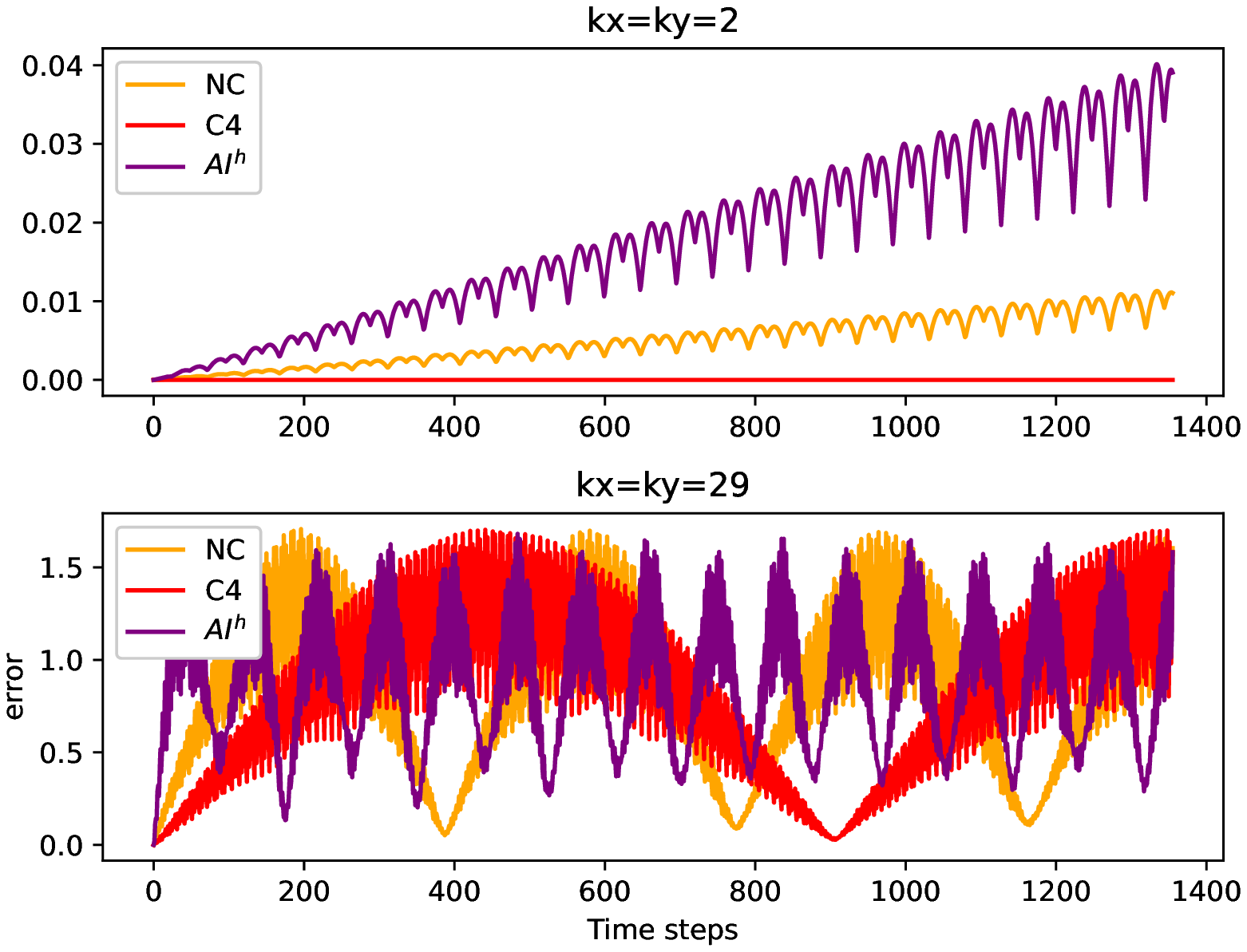}
			\end{subfigure}
			\begin{subfigure}[ht]{0.4\textwidth}
				\centering\includegraphics[scale=0.5]{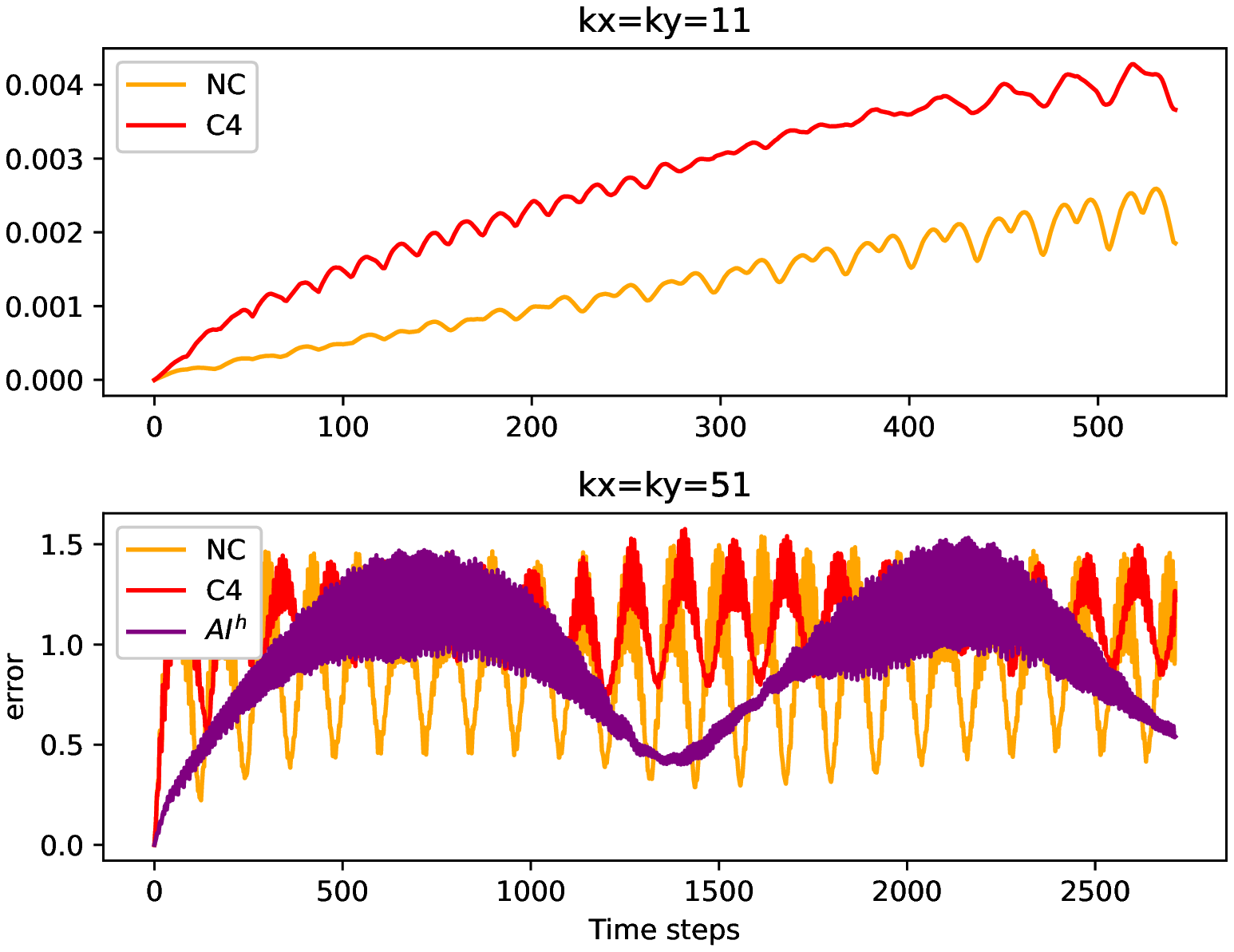}
			\end{subfigure}
			\caption{ $ h=\frac{1}{64}, r=\frac{1}{6\sqrt{2}}$, $Z=1$.
			}
			\label{fig:6}
		\end{figure}
		
		\begin{figure}[ht!]
			\begin{subfigure}[h]{0.4\textwidth}
				\centering\includegraphics[scale=0.5]{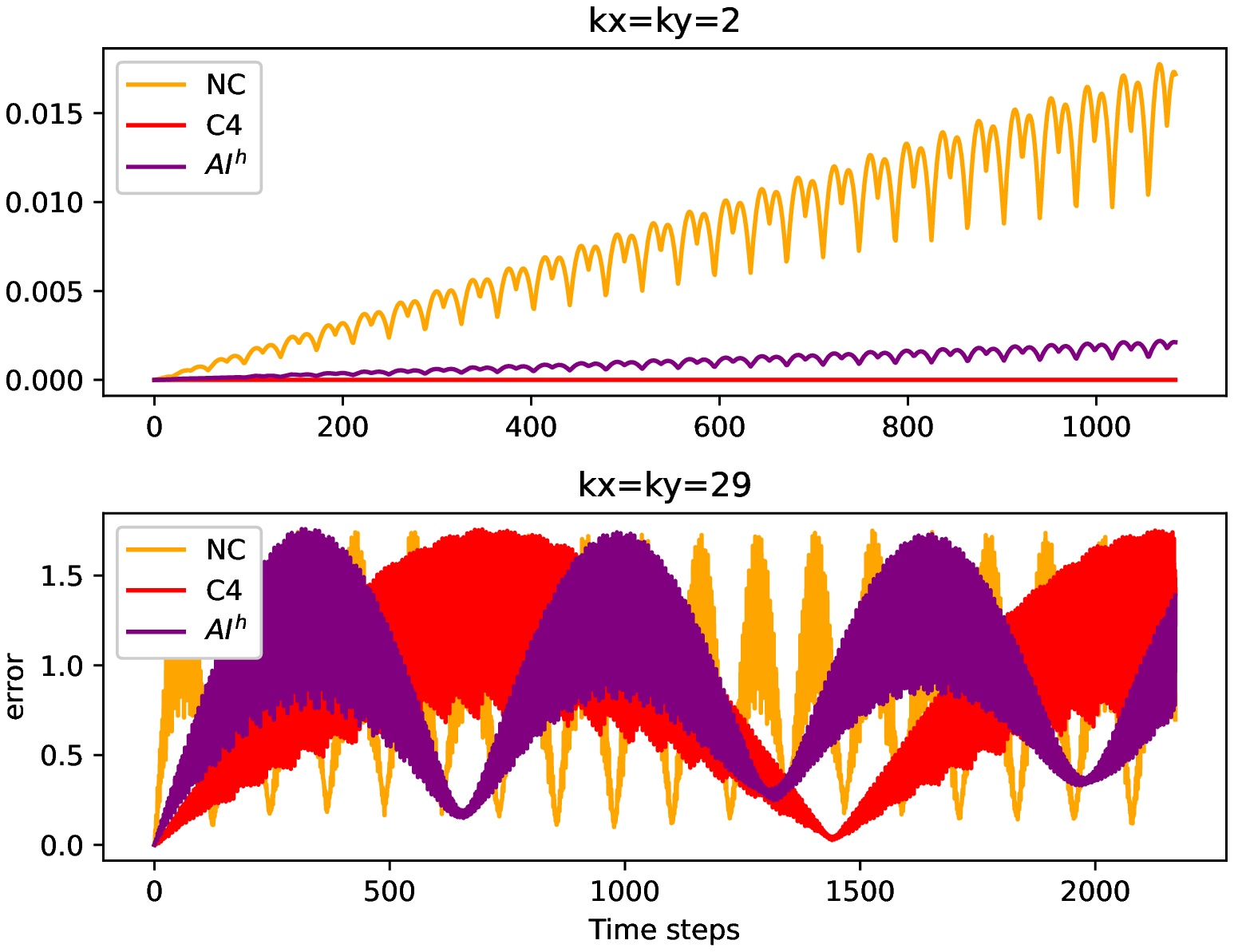}
			\end{subfigure}
			\begin{subfigure}[h]{0.4\textwidth}
				\centering\includegraphics[scale=0.5]{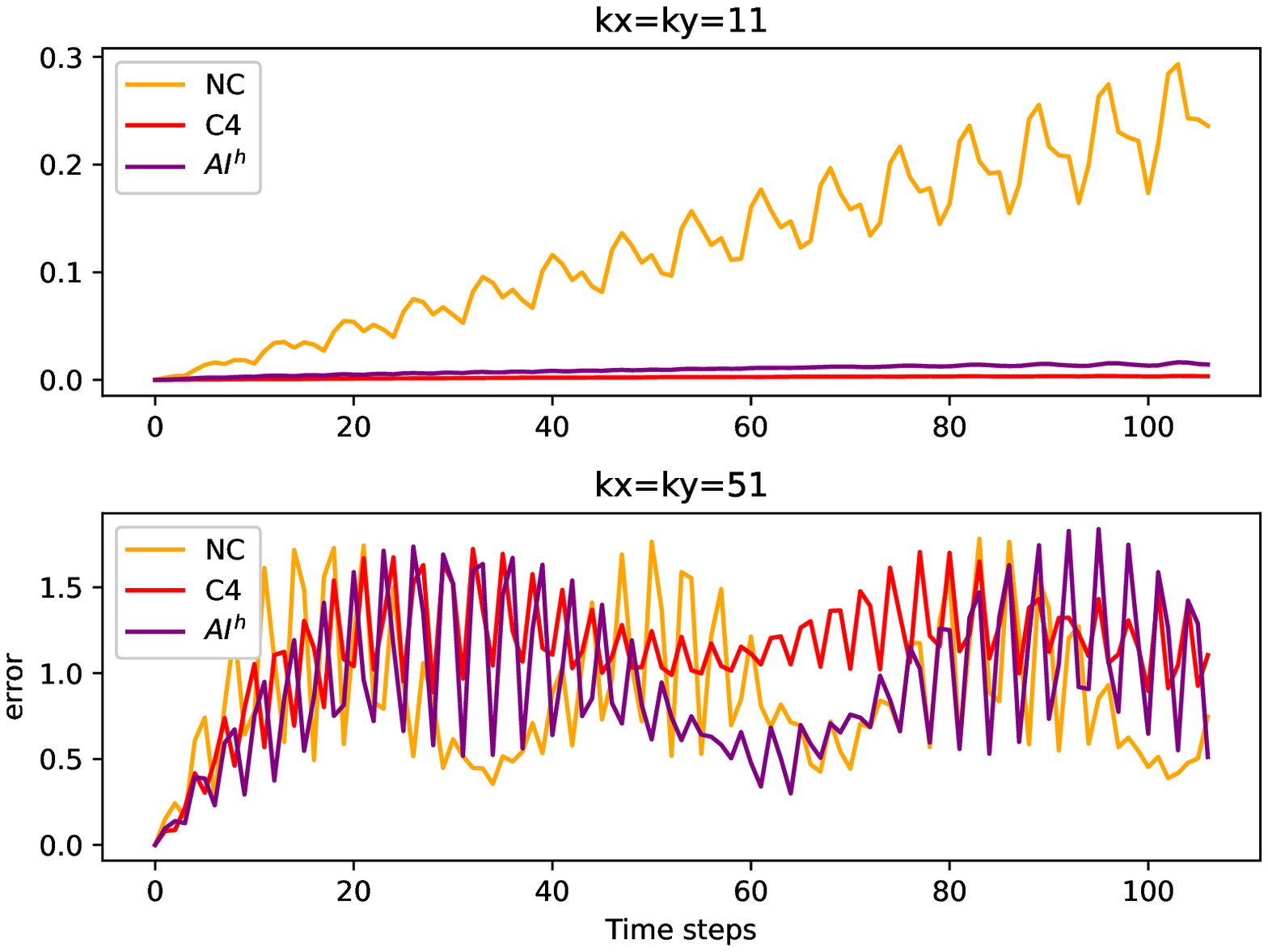}
			\end{subfigure}
			\caption{$h=\frac{1}{64}, r=\frac{5}{6\sqrt{2}}$, $Z=1$.}
			\label{fig:7}
		\end{figure}

		Moreover, we see that for $r=\frac{1}{6\sqrt{2}}$, the scheme NC is more accurate than AI while the situation is the opposite for $r=\frac{5}{6\sqrt{2}}$ .
		This is not a surprise since the dispersion effects may grow as the time step increases \cite{blinne}.

		For higher wave numbers ($\sqrt{k_x^2+k_y^2}=29$ and $\sqrt{k_x^2+k_y^2}=51$), the fourth-order accuracy of C4 and the spatial fourth-order accuracy of NC
		lose their advantage due to the pollution effect, and as the wave number increases the AI scheme demonstrates a similar
		accuracy even though it is only second-order accurate in both time and space.

		\section{Conclusions}
		We have constructed a compact implicit scheme for the 3D Maxwell's equations. The scheme is fourth-order accurate in both space and time and can therefore also be useful for long-time simulations, including scenarios where the PPW ratio is not very high.
		The scheme is compact and maintains its accuracy near the boundaries using equation-based approximations. 
		The temporal grid is staggered, which allows us
		to solve 3 scalar uncoupled elliptic equations at each half-time step. That can be done in parallel with the help of the conjugate gradient method. The elliptic equations are strictly positive definite, which enables  rapid convergence of the iterative scheme. With fixed CFL number, approximately three iterations of conjugate gradient are required for convergence regardless the grid size. This  is likely explained by the fact that the amount of positivity in equations (\ref{eq:mh}) increases as the grid length decreases because the quantity $\frac{24}{h_\tau^2}$ becomes larger as $h$ decrease and the CFL number remains fixed. Overall, the method is efficient in both CPU and memory requirements. Although we have tested numerically only 2D examples, the implementation in 3D is straightforward.
		
		As the method we have presented is finite-difference, the pollution effect cannot be avoided completely. For higher wavenumbers, lower-order methods obtained by minimizing a certain loss function over the stencil parameters (such as data-driven) can outperform higher-order schemes.
			

		\appendix
		\section{}\label{appendixa}
		
		\subsection{Operator details}\label{subsec:op_details}
		Let $\otimes$ denote the Kronecker product of matrices  
		and let $I,J,K$ be index sets.
		Let
		\begin{equation}
			D_1=
			\frac{1}{h}
			\begin{pmatrix}
				-1     & 1     & 0      & 0     & 0      & \dots  & 0      \\
				0      & -1    & 1      & 0      & 0      & \dots  & 0      \\
				0      & 0      & -1     & 1      & 0      & \dots  & 0      \\
				\vdots & \ddots & \ddots & \ddots & \ddots & \ddots & \vdots \\
				0      & 0      & \dots  & 0      & -1      & 1     & 0     \\
				0      & 0      & \dots  & 0     & 0      & -1     & 1
			\end{pmatrix}
		\end{equation}
		denote the standard central finite-difference matrix with grid size $h/2$ whose dimension equals $(m-1)\times m$, where $m\in \{|I|, |J|,|K|\}$.
		We define the following  operators  from $I\times J\times K$ to $I\times J\times K$.
		\begin{align*}
			&
			D_x:=D_1\otimes \mathbf{1}_J\otimes \mathbf{1}_K,\\&
			D_y:= \mathbf{1}_I\otimes D_1\otimes \mathbf{1}_K, \\&
			D_z:= \mathbf{1}_I\otimes \mathbf{1}_J\otimes D_1.
		\end{align*}
		
		Next, we consider a fourth-order accurate approximation for first derivatives:
		Assume that $f(x)$ is known at $N$ points $x_0,x_1, ..,x_{N-1}$. Then, one estimates $f'(x+h/2)$ at the $N-1$ points 	$x_{\frac{1}{2}},..,x_{N-3/2}$ to fourth order as:
		\begin{equation}\label{eq:lhs}
			\underbrace{
				\frac{1}{24}
				\begin{pmatrix}
					26     & -5     & 4      & -1     & 0      & \dots  & 0      \\
					1      & 22     & 1      & 0      & 0      & \dots  & 0      \\
					0      & 1      & 22     & 1      & 0      & \dots  & 0      \\
					\vdots & \ddots & \ddots & \ddots & \ddots & \ddots & \vdots \\
					0      & 0      & \dots  & 1      & 22     & 1      & 0      \\
					0      & 0      & \dots  & 0      & 1      & 22     & 1      \\
					0      & 0      & \dots  & -1     & 4      & -5     & 26
			\end{pmatrix}}_{A}
			\begin{pmatrix}
				f'(x_{\frac{1}{2}}) \\
				f'(x_{3/2}) \\
				\cdots \\
				\cdots\\
				\cdots \\
				\cdots \\
				f'(x_{N-3/2})
			\end{pmatrix}=
			\frac{1}{h}
			\begin{pmatrix}
				f(x_{1})-f(x_0)  \\
				\cdots \\
				\cdots\\
				\cdots \\
				\cdots \\
				\cdots \\
				f(x_{N-1)}-f(x_{N-2})
			\end{pmatrix}+O(h^4).
		\end{equation}
		%
		
		We define the following operators from $I\times J\times K$ to $I\times J\times K$.
		\begin{align*}
			&
			a_x^{-1}:=A^{-1}\otimes \mathbf{1}_J\otimes \mathbf{1}_K,\\&
			a_y^{-1}:= \mathbf{1}_I\otimes A^{-1}\otimes \mathbf{1}_K, \\&
			a_z^{-1}:= \mathbf{1}_I\otimes \mathbf{1}_J\otimes A^{-1},
		\end{align*}
		and
		\begin{align*}
			&
			\delta_x:=a_x^{-1}D_x,\\&
			\delta_y:=a_y^{-1}D_y, \\&
			\delta_z:=a_z^{-1}D_z.
		\end{align*}
		Finally, we define a fourth-order approximation of the curl operator as a matrix that operates on the vectorized tensor of dimension $|I|\cdot  |J| \cdot  |K|\times 1$ :
		$$
		\curl_h:=
		\begin{pmatrix}
			0& -\delta_z & \delta _y\\
			\delta_z&0&-\delta_x\\
			-\delta_y&\delta_x&0\\
		\end{pmatrix}.
		$$
		\begin{remark}\label{rem:curl estimates}
			By \cite[section 11]{GKO},
			$\|\curl_h\|\leq \frac{2\|A^{-1}\|\sqrt{3}}{h}$ and by \cite{yefet_turkel},
			$\|A^{-1}\|\sim \frac{6}{5}$.
		\end{remark}
		
		We recall the standard second-order finite difference matrices for the second derivative.
		We define the operators $D_{xx},D_{yy},D_{zz}$ from $I\times J\times K$ to
		$I\times J\times K$ using the square matrix
		$$
		D_2=
		\frac{1}{h^2}
		\begin{pmatrix}
			-2     & 1     & 0      & 0     & 0      & \dots  & 0      \\
			1     & -2     & 1      & 0      & 0      & \dots  & 0      \\
			0      & 1      & -2     & 1      & 0      & \dots  & 0      \\
			\vdots & \ddots & \ddots & \ddots & \ddots & \ddots & \vdots \\
			0      & 0      & \dots  & 1      & -2     & 1      & 0      \\
			0      & 0      & \dots  & 0      & 1      & -2    & 1      \\
			0      & 0      & \dots  & 0     & 0      & 1     & -2
		\end{pmatrix},
		$$
		\begin{align*}
			&
			D_{xx}:=D_2\otimes \mathbf{1}_J\otimes \mathbf{1}_K,\\&
			D_{yy}:= \mathbf{1}_I\otimes D_2\otimes \mathbf{1}_K, \\&
			D_{zz}:= \mathbf{1}_I\otimes \mathbf{1}_J\otimes D_2.
		\end{align*}
		
		\subsection{Operator estimates}\label{app:operatorestimates}
		Let $r=\frac{h_{\tau}}{h}$ be the  CFL number.
		Consider the finite-difference operators
		$$
		P_1:=-\left (\Delta_h+\frac{h^2}{6}\varUpsilon_h\right)+\frac{24}{h_{\tau}^2}\left (
		1+\frac{2}{r^2}
		\right),
		$$
		and
		$$
		P_2:=\frac{24}{h_{\tau}^2}\left (
		-1-\frac{2}{r^2}-\frac{h^2}{12}\Delta_h
		\right),
		$$
		where
		$$
		\Delta_h\stackrel{\text{def}}{=}D_{xx}+D_{yy}+D_{zz},
		\quad \varUpsilon_h\stackrel{\text{def}}{=} D_{zz}D_{xx}+D_{yy}D_{zz}+D_{xx}D_{yy}.
		$$
		Let $\sigma(\cdot) $ denote the spectrum of a given operator.
		The inclusion
		$$
		\sigma(h^2\Delta_h)\subset[-12,0]
		$$
		implies that
		$$
		\frac{h_{\tau}^2}{24}\sigma(-P_{2})\subset [1+2/r^2-1, 1+2/r^2].
		$$
		The operator $\frac{h_{\tau}^2}{24}\left (  -\Delta_h-\frac{h^2}{6}\varUpsilon_h \right)+1+\frac{2}{r^2}$ operates on a Fourier ansatz $\exp(\sqrt{-1}(i\theta_x+j\theta_y+k\theta_z)$ by
		\begin{align*}
			&
			1+\frac{2}{r^2}+\frac{r^2}{24}\underbrace{\left(
				6-2(\cos \theta_x+\cos \theta_y+\cos \theta_z)
				-\frac{4T}{6}
				\right)}_{S},
		\end{align*}
		where $$T=	(\cos(\theta_z)-1)(\cos(\theta_x)-1)+
		(\cos(\theta_x)-1)(\cos(\theta_y)-1)+
		(\cos(\theta_z)-1)(\cos(\theta_y)-1),$$
		and $-4\leq S\leq 12$.
		Hence,
		\begin{align}
			\begin{split}
				\frac{h_{\tau}^2}{24}\sigma(P_{1})\subset [1+2/r^2-r^2/6, 1+2/r^2+r^2/2] ,
				\\
				\frac{h_{\tau}^2}{24}	\sigma(- P_2)\in [1+\frac{2}{r^2}-1, 1+\frac{2}{r^2}].
			\end{split}
		\end{align}
		As a result, $r<\sqrt{3+\sqrt{21}}$ implies that
		$-P_2, P_1$ are positive definite and
		$$
		0<\frac{2}{r^2}\leq \|\frac{h_{\tau}^2}{24}P_{2}\|\leq 1+\frac{2}{r^2},
		$$
		and
		$$
		0<1+\frac{2}{r^2}-\frac{r^2}{6}\leq\|\frac{h_{\tau}^2}{24}P_1\|\leq 1+\frac{2}{r^2}+\frac{r^2}{2}.
		$$
		
		\subsection{The data-driven scheme} \label{subsec:AI}
		We follow the approach of \cite{ovadia}.
		The general framework for a data-driven scheme can be described as follows.
		Let $\{(X^n,Y^n)\}_{n}$ be $n$-tuples of data-points such that for any $n$,
		$X^n\in \R^m$ and $Y^n\in \R^l$.
		We define a network $\mathcal{N}_{\vec{a}}$ as a function from $\R^m\to \R^l$ which depend on parameters $\vec{a}$.
		We define a loss function of the form
		$\mathcal{L}_{\vec{a}}:=\sum_n |\mathcal{N}_{\vec{a}}(X^n)-Y^n|$
		where $|\cdot|$ is a given norm between two vectors in $\R^l$.
		We strive to solve the problem
		$$
		\underset{\vec{a}}{\mathrm{argmin}} \ \mathcal{L}_{\vec{a}}
		$$
		and then obtain the optimal parameters $\vec{a}$ and the corresponding optimal network $\mathcal{N}_{\vec{a}}$.
		The minimization process can be done using several variations of the  gradient descent  method.

		In our framework, we wish to find the optimal parameters $a,b,d$
		for evaluation of the first derivative using the stencil
		$$
		K_2=
		\frac{1}{\Delta x}
		\begin{pmatrix}
			-b & -a & a & b \\
			-d & -1+3d+2a+6b& 1-3d-2a-6b &d \\
			-b & -a & a & b
		\end{pmatrix}
		$$
		in Yee updating rules (see subsection \ref{subsec:comparble}).
		The network will be a function that takes a solution $\e^n, \h^{n+1/2}$  and returns $\e^{n+1},\h^{n+3/2}$ using Yee updating rules.
		We generate a  data set of analytical solutions using \eqref{eq:data} as follows.
		We fix $h=\frac{1}{16}$, CFL number $r$ and the corresponding $h_{\tau}$. We also fix the final time $T$. The number of spatial grid points is then denoted by $N$ and the number of times steps is denoted by $N_{\tau}$.
		For any $(k_x,k_y)$ in $\{12,13,14,15\}^2$ we generate   analytical solutions $E_z,H_x,H_y$ defined by \eqref{eq:data}. For any   $0\leq n \leq N_{\tau}$ we let
		${E_z}^n_{k_x,k_y,\mathrm{true}},{H_x}^n_{k_x,k_y\mathrm{true}},{H_y}^n_{k_x,k_y,\mathrm{true}}$ be the corresponding values of the analytical solutions evaluated at the time  $t=nh_{\tau}$.
		We let  $$X^{n,k_x,k_y}:=({E_z}^n_{k_x,k_y,\mathrm{true}},{H_x}^n_{k_x,k_y\mathrm{true}},{H_y}^n_{k_x,k_y,\mathrm{true}}).$$
		
		Next, we build the network, $\mathcal{N}_{a,b,d}$, which takes a numerical solution at time step $n$,  and outputs the solution at time $n+1$ evaluated using Yee updating rules with the stencil $K_2$.
		
		A single data point in our data set  is then defined by $(\mathrm{input}^{n,k_x.k_y}, \mathrm{output}^{n,k_x,k_y})$
		where
		$$
		\mathrm{input}^{n,k_x,k_y}=X^{n,k_x,k_y}
		$$
		and
		$$
		\mathrm{output}^{n,k_x,k_y}=(X^{n+1,k_x,k_y}, X^{n+2,k_x,k_y}, X^{n+3,k_x,k_y}).
		$$
		We omit the superscripts $k_x,k_y$ from now on.
		
		The loss function is then defined over all data points as follows:
		\begin{equation}\label{eq:er}
			\mathcal{L}_{a,b,d}=\sum_n|X^{n+1}-\mathcal{N}(X^n)|+|X^{n+2}-\mathcal{N}(\mathcal{N}(X^n))|+
			|X^{n+3}-\mathcal{N}(\mathcal{N}(\mathcal{N}(X^n)))|
		\end{equation}
		where $ |\cdot | $ in \eqref{eq:er} is taken to be the mean absolute error between two vectors.
		The loss function is minimized using Adam optimizer and Keras \cite{keras} over the data points with the usual splitting routine of train, test, and validation sets.

	\end{document}